\documentclass{amsart}
\usepackage{amsmath}
\usepackage{amssymb}
\usepackage{amsthm}
\usepackage[all]{xypic}

\theoremstyle{plain}
\newtheorem{theorem}{Theorem}
\newtheorem{proposition}[theorem]{Proposition}
\newtheorem{corollary}[theorem]{Corollary}
\newtheorem{lemma}[theorem]{Lemma}

\newtheorem{conjecture}[theorem]{Conjecture}

\newtheorem{assumptions}[theorem]{Assumptions}

\newtheorem{assnot}[theorem]{Assumption-Notation}

\theoremstyle{definition}
\newtheorem{defn}[theorem]{Definition}
\newtheorem{notation}[theorem]{Notation}

\newtheorem{example}[theorem]{Example}
\newtheorem{construction}[theorem]{Construction}

\theoremstyle{remark}

\newcommand{\beq}{\begin{equation}}
\newcommand{\eeq}{\end{equation}}

\newcommand{\ang}[1]{\langle #1 \rangle}

\newcommand{\blank}{\mbox{$\underline{\makebox[10pt]{}}$}}

\newcommand{\st}{\left\vert\right.}

\newcommand{\bbar}[1]{\overline{#1}}

\DeclareMathOperator{\id}{{Id}}
\DeclareMathOperator{\Hom}{{Hom}}

\DeclareMathOperator{\Ext}{{Ext}}

\DeclareMathOperator{\Tor}{Tor}
\DeclareMathOperator{\Aut}{{Aut}}
\DeclareMathOperator{\Proj}{Proj}
\DeclareMathOperator{\Spec}{Spec}
\DeclareMathOperator{\codim}{codim}

\DeclareMathOperator{\len}{len}
\DeclareMathOperator{\red}{red}
\DeclareMathOperator{\pd}{pd}
\DeclareMathOperator{\hd}{hd}

\DeclareMathOperator{\cohdim}{cd}
\DeclareMathOperator{\chrr}{char}

\DeclareMathOperator{\sing}{sing}
\DeclareMathOperator{\Supp}{Supp}

\newcommand{\ver}[1]{^{(#1)}}

\newcommand{\sh}{\mathcal}
\newcommand{\mf}{\mathfrak}

\newcommand{\PGL}{\PP GL}

\DeclareMathOperator{\shExt}{\mathcal{E}\!\mathit{xt}}
\DeclareMathOperator{\shTor}{\mathcal{T}\!\mathit{or}}
\renewcommand{\Lsh}{\mathcal{L}}
\newcommand{\Ish}{\mathcal{I}}
\newcommand{\struct}{\mathcal{O}}
 \newcommand{\kk}{{\Bbbk}}

\newcommand{\ZZ}{{\mathbb Z}}

\newcommand{\PP}{{\mathbb P}}
\newcommand{\NN}{{\mathbb N}}
\newcommand{\QQ}{{\mathbb Q}}

\DeclareMathOperator{\rgr}{gr-\!}
\DeclareMathOperator{\lgr}{\!-gr}

\DeclareMathOperator{\rGr}{Gr-\!}

\DeclareMathOperator{\lmod}{\!-mod}
\DeclareMathOperator{\lMod}{\!-Mod}

\DeclareMathOperator{\lQgr}{\!-Qgr}
\DeclareMathOperator{\lqgr}{\!-qgr}
\DeclareMathOperator{\lProj}{\!-Proj}
\DeclareMathOperator{\rQgr}{Qgr-\!}
\DeclareMathOperator{\rqgr}{qgr-\!}
\DeclareMathOperator{\rProj}{Proj-\!}
\DeclareMathOperator{\rTors}{Tors-\!}

\newcommand{\fracc}[2]{#1 / #2}

\newcommand{\I}{{\mathbb I}}
\newcommand{\segre}{\stackrel{s}{\otimes}}

\numberwithin{theorem}{section}        

\makeatletter                          
\let\c@equation\c@theorem              
\makeatother                           

\title{Geometric idealizers}
\author{ S. J.  Sierra}
\date{\today}
\thanks{Mathematics Department, University of Washington, Seattle, WA 98195-4350.
 { \tt sjsierra@math.washington.edu}}

\begin{document}

\begin{abstract}
Let $X$ be a projective variety, $\sigma$ an automorphism of $X$, $\Lsh$ a $\sigma$-ample invertible sheaf on $X$, and $Z$ a closed subscheme of $X$.  Inside the twisted homogeneous coordinate ring $B = B(X, \Lsh, \sigma)$, let $I$ be the right ideal of sections vanishing at $Z$. We study the subring 
\[ R = k + I\]
of $B$.  Under mild conditions on $Z$ and $\sigma$, $R$ is the {\em idealizer} of $I$ in $B$: the maximal subring of $B$ in which $I$ is a two-sided ideal.

We give geometric conditions on  $Z$ and $\sigma$ that determine the algebraic properties of $R$, and show that if $Z$ and $\sigma$ are sufficiently general, in a sense we make precise, then $R$ is left and right noetherian, has finite left and right cohomological dimension, is strongly right noetherian but not strongly left noetherian, and satisfies right $\chi_d$ (where $d = \codim Z$) but fails left $\chi_1$.    We also give an example of a right noetherian ring with infinite right cohomological dimension, partially answering a question of Stafford and Van den Bergh.
This generalizes results of Rogalski in the case that $Z$ is a point in $\PP^d$.
\end{abstract}

\maketitle

\tableofcontents


\section{Introduction}\label{IDEALIZER-INTRO}

One of the contributions of noncommutative algebraic geometry to noncommutative algebra is a gallery of interesting counterexamples:  rings with previously  unobserved, 
 unusual, and counter-intuitive properties.   Many of  these examples are even  noetherian 
  $\NN$-graded domains, but their subtler properties, such as whether or not they are  strongly noetherian or satisfy the Artin-Zhang $\chi$ conditions,  can be surprisingly pathological.   Examples of such rings include the na\"ive blowup algebras first analyzed by Rogalski in \cite{R-generic} and 
later studied in \cite{KRS}, and the idealizers constructed by Rogalski in \cite{R-idealizer}.  

In this paper we use noncommutative algebraic geometry to  construct and analyze a large class of  examples of  rings with interesting algebraic behavior.  Our rings, like those mentioned above, will be constructed as 
 subrings of twisted homogeneous coordinate rings. Here we recall Artin and Van den Bergh's construction \cite{AV}.  Let $X$ be a projective scheme, let $\Lsh$ be an invertible sheaf on $X$, and let $\sigma$ be an automorphism of $X$.  
 We denote the pullback of $\Lsh$ along $\sigma$ by 
\[ \sigma^* \Lsh = \Lsh^{\sigma}.\]  
For any $n \geq 0$, we define 
\[ \Lsh_n = \Lsh \otimes \Lsh^{\sigma} \otimes \cdots \otimes \Lsh^{\sigma^{n-1}}.\]
Then the {\em twisted homogeneous coordinate ring} $B(X, \Lsh, \sigma)$ is defined to be
\[ \bigoplus_{n \geq 0} H^0(X, \Lsh_n).\]
If $\Lsh$ is appropriately ample (the technical term is {\em $\sigma$-ample}), then $B(X, \Lsh, \sigma)$ is noetherian \cite[Theorem~1.4]{AV}.

Recall that if $I$ is  a right ideal of a ring $B$, the {\em idealizer} $\I_B(I)$ of $I$ in $B$ is the maximal subring of $B$ in which $I$ becomes a two-sided ideal.  That is,
\[ \I_B(I) = \{ x \in B \st xI \subseteq I \}.
\]
In this paper, we  study idealizer subrings of  twisted homogeneous coordinate rings.  
    Our basic construction is the following:

\begin{construction}\label{const-idealizer}\index{geometric idealizer}
Let $X$ be a  projective variety, let  $\sigma$ be an automorphism of $X$, and let $\Lsh$ be a $\sigma$-ample invertible sheaf on $X$.  
Let $Z$ be a closed subscheme of $X$. Let  $B = B(X, \Lsh, \sigma)$, and let $I$ be the right ideal of $B$ generated by sections that vanish on $Z$.  

Our object of study is the idealizer ring
\[ R= R(X, \Lsh, \sigma, Z) = \I_B(I) = \{ x \in B \st xI \subseteq I\}.\]\index{R@$R(X, \Lsh, \sigma, Z)$}
We refer to $R$ as a {\em geometric idealizer}, or more specifically, as  the {\em (right) idealizer in $B$  at $Z$}. 
\end{construction}

Our main result gives geometric criteria that determine the algebraic properties of geometric idealizers.   In particular, we characterize when such rings are noetherian.    We also analyze when idealizers are strongly noetherian, satisfy various $\chi$ conditions, and have finite cohomological dimension.   Here we define the properties that we will investigate.  

Throughout, we fix an algebraically closed ground field $\kk$.  All rings $R$ we consider will be {\em connected $\NN$-graded} $\kk$-algebras; that is, $R$ is $\NN$-graded, with $R_0 = \kk$.  

\begin{defn}\label{def-chi}
Let $R$ be a finitely generated, connected $\NN$-graded $\kk$-algebra, and let  $j \in \NN$.  We say that {\em $R$ satisfies right $\chi_j$} if, for all $i \leq j$ and for all finitely generated graded right $R$-modules $M$, we have 
\[ \dim_{\kk} \underline{\Ext}^i_R(\kk, M) < \infty.\]
We say that {\em $R$ satisfies right $\chi$} if $R$ satisfies right $\chi_j$ for all $j \in \NN$.  
We similarly define {\em left $\chi_j$} and {\em left $\chi$}; we say $R$ {\em satisfies $\chi$} if it satisfies left and right $\chi$.
\end{defn}

\begin{defn}\label{def-SN}\index{strongly noetherian}
A $\kk$-algebra $R$ is {\em strongly right (left) noetherian} if, for any commutative noetherian $\kk$-algebra $C$, the ring $R \otimes_{\kk} C$  is right (left) noetherian.
\end{defn}

Finitely generated commutative graded rings are always strongly noetherian by the Hilbert basis theorem, and  satisfy $\chi$ by \cite[Corollary~8.12]{AZ1994}.  For noncommutative rings, these two conditions are important because  their presence allows one to use powerful techniques inspired by commutative algebraic geometry.    The strong noetherian property, in particular, 
 is needed in order for modules over $R$ to be parameterized by a (commutative) projective scheme  \cite{AZ2001}.  
 If $R$ satisfies $\chi_1$, then 
one may reconstruct $R$ from the noncommutative projective scheme associated to $R$ \cite{AZ1994}.  
The higher  conditions $\chi_j$ for $j > 1$ are less well understood. However, if a ring satisfies right or left $\chi$, then  it is well-behaved in some important ways.  For example, the $\chi$ conditions are needed in order to have a version of Serre duality for a noncommutative ring $R$  \cite[Theorem 6.3]{VdB1997}  \cite[Theorem 4.2]{YZ1997}.  This is known as the existence of a {\em balanced dualizing complex}.

We show that, roughly speaking, if $\sigma$ and $Z$ are in general position (in a sense we make precise), then $R = R(X, \Lsh, \sigma, Z)$ is noetherian, strongly right noetherian, and satisfies right $\chi_d$, where $d = \codim Z$.  On the other hand, $R$
is not strongly left noetherian and fails left $\chi_1$.  
Examples are known of noetherian rings that are not strongly noetherian on one or both sides, or which satisfy $\chi_d$ on one side and fail $\chi_1$ on the other side.  However, in general, the ways in which it is possible for these  properties to fail, and to fail asymmetrically, are still poorly understood.  Thus another  goal of the work in this paper is to gain more insight into situations where some of these properties do not hold.  

 We also study the {\em cohomological dimension} of geometric idealizers, since there are many open questions about this invariant.  Recall that in noncommutative algebraic geometry, one often works with the category $\rqgr R$, defined roughly as 
\[\{ \mbox{graded right  $R$-modules} \} / \{ \mbox{finite dimensional modules} \}.\]
The category $\rqgr R$ is the appropriate analogue of $\Proj R$ for a commutative graded ring $R$.  It has a global section functor, defined as
\[ \Gamma(\sh{M}) = \Hom_{\rqgr R} ([R], \sh{M}),\]
where $[R]$ denotes the image of $R$ in $\rqgr R$.  The cohomological dimension of the functor $\Gamma$ is referred to as the {\em right cohomological dimension of $R$}; of course, one may also define {\em left cohomological dimension}.  Stafford and Van den Bergh \cite[page~194]{SV} have asked if a noetherian graded ring must have finite left and right cohomological dimension.  (This is true for commutative graded rings by a well-known result of Grothendieck \cite{Tohoku}.)  We answer Stafford and Van den Bergh's question in the affirmative for geometric idealizers.  We also  give an example of a right noetherian graded ring with infinite right cohomological dimension.   

Our work generalizes work of Rogalski \cite{R-idealizer}, who used algebraic techniques to investigate idealizers of maximal non-irrelevant graded right ideals in Zhang twists of polynomial rings.  In our language, Rogalski studied  geometric idealizers in the 
special setting that $X = \PP^d$, $\Lsh = \struct(1)$, and $Z = \{p\}$ is a point.  (His work generalized earlier work of Stafford and Zhang \cite{SZ}, who studied idealizers on $\PP^1$.)  
Rogalski  discovered that  the algebraic behavior of $R(X, \struct(1), \sigma, \{p\})$ is controlled by the geometry of the orbit of $p$.  
 
 \begin{defn}\label{idef-CD}
 Let $X$ be a variety, let  $p \in X$ and let $\sigma \in \Aut_{\kk}(X)$.  The orbit  $\{ \sigma^n(p)\}_{n \in \ZZ}$ is {\em critically dense} if it is infinite and any infinite subset  is Zariski dense in $X$.   
 \end{defn}

\begin{theorem}\label{thm-Ro}
{\em (Rogalski)}
Let $\sigma \in \PGL_{d+1}$ and let $p \in \PP^d$.
  Assume that $p$ is of infinite order under $\sigma$, and let $R =  R(\PP^d, \struct(1), \sigma, \{p\})$.  
 Then

$(1)$ $R$ is strongly right noetherian.

$(2)$ $R$ fails left $\chi_1$.

\noindent Further, if the set $\{ \sigma^n(p)\}$ is {critically dense}, then:
 
$ (3)$  $R$ is left noetherian, but not strongly left noetherian if $d \geq 2$.
 
$ (4)$ $R$ satisfies right $\chi_{d-1}$ but fails right $\chi_d$.

$(5)$  $R \otimes_{\kk} R$ is not left noetherian.   \qed
 \end{theorem}

Rogalski's work shows that  the algebraic conclusions of (3), (4), and (5) are controlled by rather subtle geometry. In particular, right idealizers at points of infinite order are automatically right noetherian, but in order for them to be left noetherian, $\sigma$ must move $p$  significantly and in some sense uniformly around $\PP^d$.  
One naturally asks if  there is a higher-dimensional  analogue of critical density that controls the behavior of more general idealizers than those studied in Theorem~\ref{thm-Ro}.   
One of the main results of this paper is that the answer is ``yes.''  We define:

\begin{defn}\label{iintro-def-HT}\index{homologically transverse}
Let $X$ be a variety and let  $Z, Y \subseteq X$ be closed subschemes.  We say that $Z$ and $Y$ are {\em homologically transverse} if
\[ \shTor^X_j (\struct_Z, \struct_Y) = 0\]
for all $j \geq 1$.
\end{defn}

\begin{defn}\label{iintro-def-CT}\index{critically transverse}
Let $X$ be a variety and let $\sigma \in \Aut_{\kk} X$.
Let $Z \subseteq X$ be a closed subscheme.  The set $\{ \sigma^n Z\}_{n \in \ZZ}$ is {\em critically transverse}  if for all closed subschemes $Y \subseteq X$, the subschemes $\sigma^n (Z)$ and $Y$ are homologically transverse  for all but finitely many $n$.
\end{defn}

In this paper, we generalize Theorem~\ref{thm-Ro} to arbitrary idealizers in twisted homogeneous coordinate rings.  We show that critical transversality  controls the behavior of these rings, and we prove:

\begin{theorem}\label{thm-idealizermain}
{\em (Theorem~\ref{thm-idealizersum})}
Let $X$ be a projective variety, let $\sigma$ be an automorphism of  $X$, and let $\Lsh$ be a $\sigma$-ample invertible sheaf on $X$.   Form the ring $R = R(X, \Lsh, \sigma, Z)$ as above.  For simplicity, assume that $Z$ is irreducible and of infinite order under $\sigma$.  (We treat the general case in the body of the paper.)   

If for all $p \in Z$, the set $\{ n \geq 0 \st \sigma^n(p) \in Z\}$ is finite, then:

$(1)$ $R$ is strongly right noetherian. 

$(2)$ $R$ fails left $\chi_1$.

\noindent If the set $\{ \sigma^n Z\}_{n \in \ZZ}$ is critically transverse, then

$ (3)$ $R$ is left noetherian, but $R$ is strongly left noetherian if and only if all components of $Z$ have codimension 1.

$(4)$  Let $d = \codim Z$.  Then $R$ fails right $\chi_d$.  If $X$ and $Z$ are smooth, then $R$ satisfies right $\chi_{d-1}$.

$(5)$ If $R$ is noetherian, then $R$ has finite left and right cohomological dimension.

$(6)$  If $Z$ is not of pure codimension 1, then $R \otimes_{\kk} R$ is not left noetherian.  
\end{theorem}

One motivation for undertaking the investigations in this paper was to make progress on the classification of noncommutative projective surfaces:  graded noetherian domains of GK-dimension 3.  Even for the nicest possible such surfaces, those that are birational to a commutative surface in the sense of \cite{Artin1995}, current classification results such as \cite{RS} depend on technical conditions such as being generated in degree 1.   
Since ideally  a classification effort in noncommutative geometry would be {\em sui generis}, without requiring such conditions,  understanding idealizers in twisted homogeneous coordinate rings has important applications to classifying noncommutative surfaces.  In a future paper, we will give a full classification of birationally commutative projective surfaces, using many of the results in the current work.

 In the remainder of the introduction, we discuss the  geometry underlying  the   technical-looking definition of critical transversality.  We first explain the use of the term ``transverse.''   Let $Y$ and $Z$ be closed subschemes of $X$.  Recall \cite[p.~427]{Ha} Serre's definition of the intersection multiplicity of $Y$ and $Z$ along the proper  component $P$ of their intersection:
\[ i( Y,Z; P) = \sum_{i \geq 0} (-1)^i \len_P (\shTor_i^X (\struct_Y, \struct_Z)),\]
where $\len_P(\sh{F})$ is the length of $\sh{F}_P$ over the local ring $\struct_{X, P}$.  Thus if  $Y$ and $Z$ are homologically transverse,  their intersection multiplicity is given by the na\"ive formula that 
\[  i(Y, Z;P) = \len_P(\struct_Y \otimes_X \struct_Z).\]
That is, $ i(Y, Z;P)$ is, as we might hope, the length of the structure sheaf of the scheme-theoretic intersection of $Y$ and $Z$ over the local ring at $P$. 

Another way of phrasing the critical transversality of $\{ \sigma^n(Z)\}$ is  that for any $Y$, the general translate of $Z$ by a power of $\sigma$ is homologically transverse to $Y$. This sort of statement is clearly reminiscent of the Kleiman-Bertini theorem, and in fact the investigations in this chapter  have led to  a new, purely algebro-geometric,  generalization of this classical result; see \cite{S-KB}.    In this paper, we apply the results in \cite{S-KB} to obtain a simple criterion for the critical transversality of $\{ \sigma^n(Z)\}$ in many cases.  

\begin{theorem}\label{ithm-whenCT}
{\em (Theorem~\ref{thm-whenCT})}
Let $\kk$ be an algebraically closed field of characteristic 0, let $X$ be a variety  over $\kk$, let $Z$ be a closed subscheme of $X$, and let $\sigma$ be an element of an algebraic group $G$ that acts on  $X$.    
  Then $\{ \sigma^n Z\}$ is critically transverse if and only if $Z$ is homologically transverse to all reduced $\sigma$-invariant subschemes of $X$.
\end{theorem}

It is reasonable to say that $\sigma$ and $Z$ are in {\em general position} if $Z$ is homologically transverse to $\sigma$-invariant subschemes of $X$.  Thus, in characteristic 0 and if $\sigma$ is an element of an algebraic group, Theorems~\ref{ithm-whenCT} and \ref{thm-idealizermain} together imply that if $\sigma$ and $Z$ are in general position, then $R(X, \Lsh,\sigma, Z)$ has the properties enumerated in Theorem~\ref{thm-idealizermain}.

{\bf Acknowledgements.}  This paper was completed as part of the author's Ph.D. thesis at the University of Michigan, under the supervision of J. T. Stafford.  During its writing, the author was partially supported  by NSF grants  DMS-0802935, DMS-0555750, and  DMS-0502170, and by a Rackham Pre-Doctoral Fellowship from the University of Michigan.  The author would like to thank Hailong Dao, Mel Hochster, and Dan Rogalski  for many informative conversations.

 \section{Definitions}\label{BACKGROUND}
 One  goal of this paper is to place the results in \cite{R-idealizer} in a more geometric context.  In that sense, this paper may be viewed as analogous to  \cite{KRS}, which defines na\"ive blowup algebras and gives a geometric construction of the algebras investigated by Rogalski in \cite{R-generic}.  As in \cite{KRS}, one of our main techniques will be to work, not with the ring $R(X, \Lsh, \sigma, Z)$, but with an associated quasi-coherent sheaf on $X$.  This object is known as a {\em bimodule algebra}, and is, roughly speaking, a sheaf with multiplicative structure.

 In this section, we give the definitions and notation to allow us to work with  bimodule algebras.  Most of the material in this section was developed in \cite{VdB1996} and  \cite{AV}, and we refer the reader there for references.  We will not work in full generality, however, and our presentation will follow that in \cite[Section~2]{KRS}.
 
 Throughout this paper, let $\kk$ be a fixed algebraically closed field; all schemes are of finite type over $\kk$.

A bimodule algebra on a variety $X$ is, roughly speaking, a quasicoherent sheaf with a multiplicative structure.  

\begin{defn}\label{def-bimod}
Let $X$ be a variety; that is, a projective integral separated scheme (of finite type over $\kk$).  An {\em $\struct_X$-bimodule}  is  a quasicoherent $\struct_{X \times X}$-module $\sh{F}$, such that for every coherent $\sh{F}' \subseteq \sh{F}$,  the projection maps $p_1, p_2: \Supp \sh{F}' \to X$ are both finite morphisms.    The left and right $\struct_X$-module structures associated to an $\struct_X$-bimodule $\sh{F}$ are defined respectively as $(p_1)_* \sh{F}$ and $ (p_2)_* \sh{F}$.  
We make the notational convention that when we refer to an $\struct_X$-bimodule simply as an $\struct_X$-module, we are using the left-handed structure (for example, when we refer to the global sections or higher cohomology
 of an $\struct_X$-bimodule).

There is a tensor product operation on the category of bimodules that has the expected properties; see \cite[Section~2]{VdB1996}.
\end{defn}

All the bimodules that we consider will be constructed from bimodules of the following form:
\begin{defn}\label{def-LR-structure}
Let $X$ be a variety and let $\sigma, \tau \in \Aut_{\kk}(X)$. Let $(\sigma, \tau)$ denote the map
\begin{align*}
X & \to X \times X \\
x & \mapsto (\sigma(x), \tau(x)).
\end{align*}
If  $\sh{F}$ is a quasicoherent sheaf on $X$, we define the $\struct_X$-bimodule ${}_{\sigma} \sh{F}_{\tau}$ to be 
\[ {}_{\sigma} \sh{F}_{\tau} = (\sigma, \tau)_* \sh{F}.\]\index{F@${}_{\sigma}\sh{F}_{\tau}$}
If $\sigma = 1$ is the identity, we will often omit it; thus we write $\sh{F}_{\tau}$ for ${}_1 \sh{F}_{\tau}$ and $\sh{F}$ for  the  $\struct_X$-bimodule ${}_1 \sh{F}_1 = \Delta_* \sh{F}$, where $\Delta: X \to X\times X$ is the diagonal.
\end{defn}

We quote a  lemma that  shows how to work with bimodules of the form ${}_{\sigma} \sh{F}_{\tau}$, and, in particular, how to form their tensor product.   
If $\sigma$ is an automorphism of $X$ and $\sh{F}$ is a sheaf on $X$, recall the notation that $\sh{F}^{\sigma} = \sigma^*\sh{F}$.\index{F@$\sh{F}^{\sigma}$}  Thus $\sigma$ acts on functions by sending $f$ to $f^\sigma = f \circ \sigma$.

\begin{lemma}\label{lem-bimods}
{\em (\cite[Lemma~2.3]{KRS})}
Let $X$ be a variety, let $\sh{F}$, $\sh{G}$ be coherent $\struct_X$-modules, and let $\sigma, \tau \in \Aut_{\kk} X$.

$(1)$ ${}_{\tau} \sh{F}_{\sigma} \cong  (\sh{F}^{\tau^{-1}})_{\sigma \tau^{-1}}$.

$(2)$  $\sh{F}_{\sigma} \otimes \sh{G}_{\tau} \cong (\sh{F} \otimes \sh{G}^{\sigma})_{\tau \sigma}$.

$(3)$ In particular, if $\Lsh$ is an invertible sheaf on $X$, then  $\Lsh_{\sigma}^{\otimes n} = (\Lsh_n)_{\sigma^n}$.  \qed
\end{lemma}

\begin{defn}\label{def-BMA}
Let $X$ be a variety and let $\sigma \in \Aut_{\kk} X$.  
An {\em $\struct_X$-bimodule algebra}, or simply a {\em bimodule algebra}, $\sh{B}$ is an algebra object in the category of bimodules.  That is, there is a unit map $1: \struct_X \to \sh{B}$ and a product map $\mu: \sh{B} \otimes \sh{B} \to \sh{B}$ that have the usual properties.
\end{defn}

We follow \cite{KRS} and define
\begin{defn}\label{def-gradedBMA}\index{graded bimodule algebra}
Let $X$ be a variety and let $\sigma \in \Aut_{\kk} X$.  
A bimodule algebra $\sh{B}$ is a {\em graded $(\struct_X, \sigma)$-bimodule algebra} if:

(1) There are coherent sheaves $\sh{B}_n$ on $X$ such that
\[ \sh{B} = \bigoplus_{n \in \ZZ} {}_1(\sh{B}_n)_{\sigma^n};\]

(2)  $\sh{B}_0 = \struct_X $;

(3) the multiplication map $\mu$ is given by $\struct_X$-module maps
$\sh{B}_n \otimes \sh{B}_m^{\sigma^n} \to \sh{B}_{n+m}$, satisfying the obvious associativity conditions.  
\end{defn}

\begin{defn}\label{def-module}\index{right $\sh{B}$-module}
Let $X$ be a variety and let $\sigma \in \Aut_{\kk} X$.  
Let $\sh{R}$ be a graded $(\struct_X, \sigma)$-bimodule algebra.  A {\em right $\sh{R}$-module} $\sh{M}$ is a quasicoherent $\struct_X$-module $\sh{M}$ together with a right $\struct_X$-module map $\mu: \sh{M} \otimes \sh{R} \to \sh{M}$ satisfying the usual axioms.  We say that $\sh{M}$ is {\em graded} if there is a direct sum decomposition 
\[\sh{M} = \bigoplus_{n \in \ZZ}  (\sh{M}_n)_{\sigma^n}\]
 with multiplication giving a family of $\struct_X$-module maps $\sh{M}_n \otimes \sh{R}_m^{\sigma^n} \to \sh{M}_{n+m}$, obeying the appropriate axioms.
 
  We say that $\sh{M}$ is {\em coherent} if there are a coherent $\struct_X$-module $\sh{M}'$ and a surjective map $\sh{M}' \otimes \sh{R} \to \sh{M}$ of ungraded $\struct_X$-modules.  
We make similar definitions for left $\sh{R}$-modules.  The bimodule algebra $\sh{R}$ is {\em right (left) noetherian}
if every right (left) ideal of $\sh{R}$ is coherent.  By standard arguments, a graded $(\struct_X, \sigma)$-bimodule algebra is right (left) noetherian if and only if every graded right (left) ideal is coherent.
\end{defn}

We alert the reader to a technicality of notation.  Suppose that $\sh{R}$ is a bimodule algebra on $X$,  that $\sh{M}$ is a right $\sh{R}$-module, and that $\sh{F}$ is a subsheaf of $\sh{M}$.  We will denote the image of $\sh{F} \otimes \sh{R}$ under the multiplication map 
\[ \xymatrix{ 
\sh{F} \otimes \sh{R} \ar[r] &  \sh{M} \otimes \sh{R} \ar[r]^{\mu} & \sh{M}}\]
by $\sh{F} \cdot \sh{R}$.  In the case that $\sh{F} \subseteq \struct_X$ is also an ideal sheaf on $X$, to avoid ambiguity we will denote the image of $\sh{F} \otimes \sh{R}$ under the canonical map
\[ 
\xymatrix{ \sh{F} \otimes \sh{R} \ar[r] &  \struct_X \otimes \sh{R} = \sh{R}}\]
by $\sh{FR}$.

Coherence for $\sh{R}$-modules should be viewed as analogous to finite generation, but it is unknown whether, for a  general noetherian bimodule algebra, every submodule of a coherent module is coherent!   Fortunately, in our situation  the usual intuitions do hold.  We restate \cite[Proposition~2.10]{KRS} as:

\begin{lemma}\label{lem-noeth}
Let $X$ be a variety and let $\sigma \in \Aut_{\kk} X$.  
Let $\sh{R} = \bigoplus_{n \in \ZZ} (\sh{R}_n)_{\sigma^n}$ be a graded $(\struct_X, \sigma)$-bimodule algebra  such that each $\sh{R}_n$ is a subsheaf of an invertible sheaf.  Then
$\mathcal{R}$ is right (left) noetherian if and only if all submodules of coherent right (left) $\mathcal{R}$-modules are coherent. \qed
\end{lemma}

We recall here some standard notation on module categories over rings and bimodule algebras.  Let $R$ be an $\NN$-graded ring.  We define $\rGr R$ to be the category of $\ZZ$-graded right $R$-modules; morphisms in $\rGr R$ preserve degree.   Let $\rTors R$ be the full subcategory of modules that are direct limits of finite-dimensional modules.  This is a Serre subcategory of $\rGr R$, so we may form the {\em quotient category}
\[ \rQgr R := \rGr R / \rTors R.\]
(We refer the reader to \cite{Gabriel1962} as a reference for the category theory used here.)  There is a canonical {\em quotient functor} $\pi: \rGr R \to \rQgr R$.  

We make similar definitions on the left.   Further, throughout this paper, we adopt the convention that if Xyz is a category, then xyz is the full subcategory of noetherian objects.  Thus we have $\rgr R$ and $\rqgr R = \pi (\rgr R)$, $R \lqgr$, etc.  If $X$ is  a variety, we will denote the category of quasicoherent (respectively coherent) sheaves on $X$ by $\struct_X \lMod$ (respectively $\struct_X \lmod$).  

Given a module $M \in \rgr R$, we define $M[n] = \bigoplus_{n \in \ZZ} M[n]_i$, where
\[ M[n]_i = M_{n+i}.\]
If $M, N \in \rgr R$, let 
\[ \underline{\Hom}_{\rgr R} (M, N) = \bigoplus_{n \in \ZZ} \Hom_{\rgr R}(M, N [n]).\]
Similarly, if $\sh{M}, \sh{N} \in \rqgr R$, we define
\[ \underline{\Hom}_{\rqgr R} (\sh{M}, \sh{N}) = \bigoplus_{n \in \ZZ} \Hom_{\rqgr R}(\sh{M}, \sh{N} [n]).\]
The $\underline{\Hom}$ functors have derived functors  $\underline{\Ext}_{\rgr R}$ and $\underline{\Ext}_{\rqgr R}$, defined in the obvious way.  

For a graded $(\struct_X, \sigma)$-bimodule algebra $\sh{R}$, we likewise define $\rGr \sh{R}$ and $\rgr \sh{R}$.  The full subcategory  $\rTors \sh{R}$ of $\rGr \sh{R}$ consists of direct limits of modules that are  coherent as $\struct_X$-modules, and we similarly define
\[ \rQgr \sh{R} := \rGr \sh{R}/\rTors \sh{R}.\]
We define $\rqgr \sh{R}$ in the obvious way.

If $\sh{R}$ is a graded $(\struct_X, \sigma)$-bimodule algebra, we may form its {\em section algebra}
\[ H^0(X, \sh{R})  = \bigoplus_{n\geq 0}H^0(X, \sh{R}_n).\]
Multiplication on $H^0 (X, \sh{R})$ is induced from the multiplication map $\mu$ on $\sh{R}$; that is, from the maps
\[ \xymatrix{
H^0(X, \sh{R}_n) \otimes H^0(X, \sh{R}_m) \ar[r]^{1\otimes \sigma^n} 
&  H^0(X, \sh{R}_n) \otimes H^0(X, \sh{R}_m^{\sigma^n}) \ar[r]^<<<<<{\mu} 
& H^0(X, \sh{R}_{n+m}).
}\]

  If $\sh{M}$ is a graded right $\sh{R}$-module, then 
\[ H^0(X, \sh{M}) = \bigoplus_{n \in \ZZ} H^0(X, \sh{M}_n)\]  
 is a right $H^0(X, \sh{R})$-module in the obvious way; thus $H^0(X, \blank)$ is a functor from $\rGr \sh{R}$ to $\rGr H^0(X, \sh{R})$.  

 If  $R = H^0(X, \sh{R})$ and $M$ is  a graded right $R$-module, define 
$M \otimes_R \sh{R}$ to be the sheaf associated to the presheaf $V \mapsto M \otimes_R \sh{R}(V)$.   This is  a graded right $\sh{R}$-module, and the functor $\blank \otimes_R \sh{R}: \rGr R \to \rGr \sh{R}$ is a right adjoint to $H^0(X, \blank)$.

The fundamental result on when one can more closely relate $\rGr \sh{R}$ and $\rGr R$ is due to Van den Bergh.  We first give a definition:

\begin{defn}\label{def-ample}\index{ample@(left, right) ample sequence}
Let $X$ be a projective variety, let $\sigma \in \Aut_{\kk} X$, and let  $\{\sh{R}_n\}_{n \in \NN}$ be a sequence of coherent sheaves  on $X$.  The sequence of bimodules $\{(\sh{R}_n)_{\sigma^n}\}_{n \in \NN}$ is {\em right ample} if for any coherent $\struct_X$-module $\sh{F}$, the following properties hold:
 
 (i)   $\sh{F} \otimes \sh{R}_n$ is globally generated for $n \gg 0$;
 
 (ii) $H^q(X, \sh{F} \otimes \sh{R}_n) = 0$  for $n \gg 0$ and all $q \geq 1$.  
 
\noindent The sequence $\{ (\sh{R}_n)_{\sigma^n}\}_{n \in \NN}$  is {\em left ample}  if for any coherent $\struct_X$-module $\sh{F}$, the following properties hold:
 
 (i)  $\sh{R}_n \otimes \sh{F}^{\sigma^n}$ is globally generated for $n \gg 0$;
 
 (ii) $H^q(X, \sh{R}_n \otimes \sh{F}^{\sigma^n}) = 0$  for $n \gg 0$ and all $q \geq 1$.

We say that an invertible  sheaf $\Lsh$ is {\em $\sigma$-ample}\index{sigma@$\sigma$-ample} if the $\struct_X$-bimodules 
\[\{(\Lsh_n)_{\sigma^n}\}_{n \in \NN} = \{ \Lsh_\sigma^{\otimes n}\}_{n \in \NN}\]
 form a right ample sequence.  By  \cite[Theorem~1.2]{Keeler2000}, this is true if and only if the $\struct_X$-bimodules $\{(\Lsh_n)_{\sigma^n}\}_{n \in \NN}$ form a left ample sequence.
\end{defn}

The following result is a special case of a result due to Van den Bergh \cite[Theorem~5.2]{VdB1996}, although we follow the presentation of  \cite[Theorem~2.12]{KRS}:

\begin{theorem}\label{thm-VdBSerre}
{\em (Van den Bergh)}
Let $X$ be a projective scheme and let $\sigma$ be an automorphism of $X$.  Let $\sh{R} = \bigoplus (\sh{R}_n)_{\sigma^n}$ be a right noetherian graded $(\struct_X, \sigma)$-bimodule algebra, such that the bimodules  $\{(\sh{R}_n)_{\sigma^n}\}$ form a right ample sequence.   Then $R = H^0(X, \sh{R})$ is also right noetherian, and the functors $H^0(X, \blank)$ and $\blank \otimes_R \sh{R}$ induce an equivalence of categories
\[ \rqgr \sh{R} \simeq \rqgr R.\] 
\qed
\end{theorem}

\begin{example}[Twisted bimodule algebras]\index{twisted bimodule algebra}\index{B@$\sh{B}(X, \Lsh, \sigma)$}
\label{eg-twist}
Let $X$ be a projective scheme, let $\sigma \in \Aut_{\kk}(X)$, and let $\Lsh$ be an invertible sheaf on $X$.  We define the {\em twisted bimodule algebra of $\Lsh$} to be
\[ \sh{B} = \sh{B}(X, \Lsh, \sigma) = \bigoplus_{n \geq 0} \Lsh_n.\]
Then $\sh{B}$ is an $(\struct_X, \sigma)$-graded bimodule algebra.  
Taking global sections of $\sh{B}(X, \Lsh, \sigma)$ gives the twisted homogeneous coordinate ring $B(X, \Lsh, \sigma)$.
 
Assume now that $\Lsh$ is $\sigma$-ample.   We record here some important properties of $B$ and $\sh{B}$.
By Theorem~\ref{thm-VdBSerre}, the categories $\rqgr \sh{B}$ and $\rqgr B$ are equivalent; this was originally proved by Artin and Van den Bergh \cite{AV}, who also observed that both categories are equivalent to $\struct_X \lmod$.  The equivalence between 
between $\rqgr B$ and $\struct_X \lmod$ is given as follows.  Define a functor
\begin{align*}
\Gamma_*: \struct_X \lmod & \to \rgr B\\
	\sh{F}	& \mapsto \bigoplus_{n \geq 0} H^0(X, \sh{F} \otimes \Lsh_n).
\end{align*}\index{Gamma@$\Gamma_*$}
The composition $\pi \Gamma_*$ has a quasi-inverse, induced by the 
 functor
\[ \widetilde{}\,\,:  \rgr B \to \struct_X \lmod.\]
To define this functor, let $M \in \rgr B$.  There is a unique coherent sheaf $\sh{F}$ such that
$\sh{F} \otimes \Lsh_n = (M \otimes_B \sh{B})_n$ for all $n \gg 0$.  Define $\widetilde{M} = \sh{F}$.

Since $\sigma$-ampleness is left-right symmetric,  there is also an  equivalence $B \lqgr \simeq \struct_X \lmod$.  The quasi-inverses between these two categories are defined by letting
\[\Gamma_* (\sh{F}) = \bigoplus_{n \geq 0} H^0(\Lsh_n \otimes \sh{F}^{\sigma^n}) \in B \lgr \]
and letting  $\widetilde{M}$ be the unique $\sh{F}$ such that $\Lsh_n \otimes \sh{F}^{\sigma^n} = (\sh{B} \otimes _B M)_n$ for all $n \gg 0$.

We note that the shift functors $(\blank)[n]$ on $\rgr B$ induce autoequivalences of $\struct_X \lmod$.  By \cite[(3.1)]{SV}, if  $N \in \rgr B$, then 
\beq\label{CHARLIE}
 \widetilde{N[n]} \cong (\widetilde{N} \otimes \Lsh_n)^{\sigma^{-n}}
 \eeq
for all $n \geq 0$.
\end{example}

Throughout this paper, we will work with sub-bimodule algebras of 
the twisted bimodule algebra $\sh{B} =\sh{B}(X, \Lsh, \sigma)$.  We note here that  the invertible sheaf $\Lsh$ makes only a formal difference.  

\begin{lemma}\label{lem-catequiv}
Let $X$ be a projective scheme with automorphism $\sigma$, and let $\Lsh$ be an invertible sheaf on $X$.  Let 
\[ \sh{R} = \bigoplus_{n \geq 0} (\sh{R}_n)_{\sigma^n}\]
 be a graded $(\struct_X,\sigma)$-sub-bimodule algebra of the twisted bimodule algebra $\sh{B}(X, \Lsh, \sigma)$.  Let $\sh{S}_n = \sh{R}_n \otimes  \Lsh_n^{-1}$ for $n \geq 0$.

 $(1)$ Let $\sh{S}$ be the graded $(\struct_X, \sigma)$-bimodule algebra defined by 
 \[ \sh{S} = \bigoplus_{n \geq 0} (\sh{S}_n)_{\sigma^n} .\]
 Then the categories $\rgr \sh{R}$ and $\rgr \sh{S}$ are equivalent, and the categories $\sh{S} \lgr$ and $\sh{R} \lgr $ are equivalent.
 
$ (2)$ Let $\sh{H}$ be an invertible sheaf on $X$ and let $k \in \ZZ$.  Then the functor $\sh{H}_{\sigma^k} \otimes \blank$ that maps
 \[ \sh{M} = \bigoplus_{n \in \ZZ} (\sh{M}_n)_{\sigma^n}  \mapsto \bigoplus_{n \in \ZZ} ( \sh{H} \otimes \sh{M}_n^{\sigma^k})_{\sigma^{k+n} }= \sh{H}_{\sigma^k} \otimes \sh{M}\]
 is an autoequivalence of $\rgr \sh{R}$.
 \end{lemma}
 \begin{proof}
 (1)  By symmetry, it suffices to prove that $\rgr \sh{R} \simeq \rgr \sh{S}$.  
 For $n < 0$, define 
 \[\Lsh_n = (\Lsh^{-1})^{\sigma^n} \otimes (\Lsh^{-1})^{\sigma^{n+1}} \otimes \cdots \otimes (\Lsh^{-1})^{\sigma^{-1}}.\]
 Define a functor 
 $ F: \rgr \sh{R} \to \rgr \sh{S}$
 as follows:
 if 
 \[ \sh{M} = \bigoplus_{n \in \ZZ} (\sh{M}_n)_{\sigma^n}\]
 is a graded right $\sh{R}$-module, 
 define
  \[ F(\sh{M}) = \bigoplus_{n \in \ZZ} (\sh{M}_n \otimes (\Lsh_n)^{-1})_{\sigma^n}.\]
  The inverse functor $G: \rgr \sh{S} \to \rgr \sh{R}$ is defined as follows:  if 
  \[ \sh{N} = \bigoplus_{n \in \ZZ} (\sh{N}_n)_{\sigma^n}\]
  is a graded right $\sh{S}$-module, let
  \[ G(\sh{N}) = \bigoplus_{n \in \ZZ} (\sh{N}_n \otimes \Lsh_n)_{\sigma^n}.\]
  It is trivial that $GF \cong \id_{\rgr \sh{R}}$ and that $FG \cong \id_{\rgr \sh{S}}$.
  
  (2)  By Lemma~\ref{lem-bimods}(2), we have that
  \[ \bigl((\sh{H}^{\sigma^{-k}})^{-1} \bigr)_{\sigma^{-k}} \otimes \sh{H}_{\sigma^k} \cong {}_{1} \bigl( (\sh{H}^{\sigma^{-k}})^{-1} \otimes \sh{H}^{\sigma^{-k}})_1 \cong \struct_X.\]
Thus  the functor $((\sh{H}^{\sigma^{-k}})^{-1})_{\sigma^{-k}} \otimes \blank$ is a quasi-inverse to $\sh{H}_{\sigma^k} \otimes \blank$.  
  \end{proof}

  To end this section, we give the sheaf-theoretic versions of some standard results on primary decomposition of ideals in a commutative ring.  
Let $X$ be a locally noetherian scheme and let  $\sh{I}$ be a proper ideal sheaf  on $X$.  We will say that $\sh{I}$ is {\em prime}\index{prime ideal sheaf} if it defines a reduced and irreducible subscheme of $X$.    
We say that $\sh{I}$ is {\em $\sh{P}$-primary}\index{primary ideal sheaf} if there is a prime ideal sheaf $\sh{P}$ such that some $\sh{P}^n \subseteq \sh{I}$, and for all ideal sheaves 
$\sh{J}$ and  $\sh{K}$ on $X$, if  $\sh{J K} \subseteq \sh{I}$ but $\sh{J} \not\subseteq \sh{P}$, then  $\sh{K} \subseteq \sh{I}$.

Since primary decompositions localize, the  theory of primary decomposition of ideals in  a commutative ring translates straightforwardly to ideal sheaves on $X$.   In particular, any ideal sheaf $\sh{I}$ has a {\em minimal primary decomposition}
\[\sh{I} = \sh{I}_1 \cap \cdots \cap \sh{I}_c,\]
where each $\sh{I}_i$ is $\sh{P}_i$-primary for some prime ideal sheaf $\sh{P}_i$, the $\sh{P}_i$ are all distinct, and $\sh{I}$ may not be written as an intersection with fewer terms.   If $\sh{P}_i$ is a minimal prime over $\sh{I}$, then  we will refer to $\sh{I}_i$ as a {\em minimal primary component}\index{primary component!minimal} of $\sh{I}$.  If $\sh{P}_i$ is not minimal over $\sh{I}$, we will refer to $\sh{I}_i$ as an {\em embedded primary component}\index{primary component!embedded}.  As is well-known, the primes $\sh{P}_i$ and the minimal primary components of $\sh{I}$ are uniquely determined by $\sh{I}$, while the embedded primary components are not necessarily unique.  

Now let $Z$ be  a closed subscheme of $X$ and let $\sh{I}$ be the ideal sheaf of $Z$.  
Let $\sh{I} = \sh{I}_1 \cap \cdots \cap \sh{I}_c$ be a minimal primary decomposition of $\sh{I}$.  We will refer to the closed subschemes $Z_i$ defined by the minimal primary components $\sh{I}_i$ of $\sh{I}$ as the {\em irreducible components}\index{component!irreducible} of $Z$.  We will refer to the subschemes defined by embedded primary components as {\em embedded components}\index{component!embedded} of $Z$.  Together, the irreducible and embedded components make up the {\em primary components}\index{component!primary} of $Z$.   

 For any two ideal sheaves $\sh{K}$ and $\sh{J}$ on $X$, we define  the {\em ideal quotient}
 \[ (\sh{J}: \sh{K})\]
to be the maximal coherent subsheaf $\sh{F}$ of $\struct_X$ such that $\sh{KF} \subseteq \sh{J}$. 
We record the following elementary lemmas for future use.

\begin{lemma}\label{lem-primary}
 Let $\sh{I} = \sh{I}_1 \cap \cdots \cap \sh{I}_c$ be a primary decomposition of the ideal sheaf $\sh{I}$ on the locally noetherian scheme $X$, where each $\sh{I}_i$ is $\sh{Q}_i$-primary for some prime ideal sheaf $\sh{Q}_i$.

$(1)$  If $\sh{K}$ and $\sh{J}$ are ideal sheaves so that $\sh{K} \not\subseteq \sh{Q}_i$ for some $i$, then
\[ \bigl( \sh{I}: (\sh{K} \cap \sh{J}) \bigr) \subseteq (\sh{I}_i : \sh{J}).\]

$(2)$  If $\sh{K}$ is not contained in any $\sh{Q}_i$, then $(\sh{I}: \sh{K}) = \sh{I}$.
\end{lemma}
\begin{proof}
(1)  We have
\[ 
\bigl( \sh{I}: (\sh{K} \cap \sh{J}) \bigr) \sh{JK} \subseteq  \bigl( \sh{I}: (\sh{K} \cap \sh{J}) \bigr) (\sh{K} \cap \sh{J}) \subseteq \sh{I} \subseteq \sh{I}_i.
\]
As $\sh{K} \not \subseteq \sh{Q}_i$ and $\sh{I}_i$ is $\sh{Q}_i$-primary, 
$\bigl( \sh{I}: (\sh{K} \cap \sh{J}) \bigr)  \sh{J} \subseteq \sh{I}_i$
and so by definition
\[ 
\bigl( \sh{I}: (\sh{K} \cap \sh{J}) \bigr)  \subseteq (\sh{I}_i: \sh{J}).\]

(2)  Applying (1) with $\sh{J}  =\struct_X$, we see that
\[ 
(\sh{I}: \sh{K}) \subseteq \bigcap_{i=1}^c (\sh{I}_i: \struct_X)  = \bigcap_{i=1}^c \sh{I}_i = \sh{I}.
\]
The other containment is automatic.
\end{proof}

\begin{lemma}\label{lem-prim2}
Let $\sh{P}$ and $\sh{I}$ be ideal sheaves on the locally noetherian scheme $X$, where $\sh{P}$ is prime and $\sh{I}$ is $\sh{P}$-primary.  If $\sh{J}$ is an ideal sheaf on $X$ that is not contained in $\sh{I}$, then $(\sh{I}:\sh{J})$ is also $\sh{P}$-primary. 
\end{lemma}
\begin{proof} 
Since $\sh{J} \not\subseteq \sh{I}$, we have that $(\sh{I}:\sh{J}) \neq \struct_X$.  Suppose that $\sh{F}$ and $\sh{G}$ are ideal sheaves with $\sh{F} \not\subseteq \sh{P}$ and $\sh{FG} \subseteq (\sh{I}:\sh{J})$.  Thus $\sh{FGJ} \subseteq \sh{I}$, and since $\sh{I}$ is $\sh{P}$-primary,  $\sh{GJ} \subseteq \sh{I}$.  This precisely says that $\sh{G} \subseteq (\sh{I}:\sh{J})$.  
Since for some $m$, we have $\sh{P}^m \subseteq \sh{I} \subseteq (\sh{I}:\sh{J})$, we see that $(\sh{I}:\sh{J})$ is $\sh{P}$-primary.
\end{proof}


\section{Right  noetherian bimodule algebras}\label{RIGHT-NOETH-BMA}


Let $X$, $\Lsh$, $\sigma$, and $Z$ be as in Construction~\ref{const-idealizer}, and let  $R$ be the geometric idealizer ring 
\[ R= R(X, \Lsh, \sigma, Z).\]
In this section, we begin  the  study of $R$ by working with the corresponding  bimodule algebra.  Here we introduce the notation we will use.    
\begin{notation}\label{notnot}
Let $X$ be a projective variety, let $\sigma \in \Aut_{\kk} X$, and let $\Lsh$ be an invertible sheaf on $X$.  Let $Z$ be a closed subscheme of $X$ and let $\sh{I} = \sh{I}_Z$ be its defining ideal.  Let
\[\sh{B} = \sh{B}(X, \Lsh, \sigma) = \bigoplus_{n \geq 0} (\Lsh_n)_{\sigma^n},\]
and let $\sh{R} = \sh{R}(X, \Lsh, \sigma, Z)$ be the graded $(\struct_X, \sigma)$-sub-bimodule algebra of $\sh{B}$ defined by 
\[ \sh{R} = \sh{R}(X, \Lsh, \sigma, Z) = \bigoplus_{n \geq 0} \bigl( ( \sh{I}: \sh{I}^{\sigma^n}) \Lsh_n \bigr)_{\sigma^n}.\]\index{R@$\sh{R}(X, \Lsh, \sigma, Z)$}

It is easy to see that $\sh{R}$ is the maximal sub-bimodule algebra of $\sh{B}$ such that $\sh{IB}_+$ is a two-sided ideal of $\sh{R}$, and we  will write
\[ \sh{R} = \I_{\sh{B}} (\sh{IB}_+)\]
and speak of $\sh{R}$ as an {\em idealizer bimodule algebra}\index{idealizer bimodule algebra} inside $\sh{B}$. As usual we write 
\[\sh{R} = \bigoplus_{n \geq 0} (\sh{R}_n)_{\sigma^n},\]
so 
\[ \sh{R}_n = (\sh{I}: \sh{I}^{\sigma^n}) \Lsh_n.\]
\end{notation}

In the next two  sections, we give geometric conditions on the defining data $(X, \Lsh, \sigma, Z)$ that determine when $\sh{R}(X, \Lsh, \sigma, Z)$ is left or right noetherian.  
In this section, we consider 
 when $\sh{R}(X, \Lsh, \sigma, Z)$ is right noetherian; we will show this is controlled by a straightforward geometric property of the intersection of $Z$ with $\sigma$-orbits.

\begin{defn}\label{def-forwardorbit}\index{forward orbits, finite intersection with}
Let $x \in X$ and let $\sigma$ be an automorphism of $X$.  The {\em forward $\sigma$-orbit} or {\em forward orbit} of $x$ is the set 
\[ \{ \sigma^n(x) \st n \geq 0 \}.\]
If $Z \subset X$ is such that for any $x \in X$, the set
\[ \{ n \geq 0 \st \sigma^n(x) \in Z \}\]
is finite, we say that {\em $Z$ has finite intersection with forward orbits}.   In particular, if $Z$ has finite intersection with forward orbits, it contains no points of finite order under $\sigma$.
\end{defn}

 Recall that if  $R$ is a $\ZZ$-graded or $\NN$-graded ring, and   $1 \leq n \in \NN$, then the {\em $n$th Veronese of $R$}  is the graded ring $R \ver{n}$ defined by 
  \[ R \ver{n}_i = R_{ni}.\]
We note here that 
\[ \sh{B}(X, \Lsh, \sigma) \ver{n} \cong \sh{B}(X, \Lsh_n, \sigma^n)\]
and that 
\[ \sh{R}(X, \Lsh, \sigma, Z) \ver{n}  \cong \sh{R}(X, \Lsh_n, \sigma^n, Z).\]

We will show:  

\begin{theorem}\label{thm-rtnoeth}
Assume Notation~\ref{notnot}.  Let 
\beq\label{kiri2}
 \sh{I} = \sh{J}_1 \cap \cdots \cap \sh{J}_c \cap \sh{K}_1 \cap \cdots \cap \sh{K}_e
\eeq
be a minimal primary decomposition of $\sh{I}$, where each $\sh{J}_i$ is $\sh{P}_i$-primary for some  prime ideal sheaf $\sh{P}_i$ of finite order under $\sigma$, and each $\sh{K}_j$ is $\sh{Q}_j$-primary for some prime ideal sheaf $\sh{Q}_j$ of infinite order under $\sigma$.  Let $W$ be the closed subscheme of $Z$ defined by the ideal sheaf $\sh{K} =  \sh{K}_1 \cap \cdots \cap \sh{K}_e$, and let $\sh{J} = \sh{J}_1 \cap \cdots \cap \sh{J}_c$.  Then the following are equivalent:

$(1)$ $\sh{R} = \sh{R}(X, \Lsh, \sigma, Z)$ is right noetherian;

$(2)$ there is some $n$ so that $\sh{J}^{\sigma^n} = \sh{J}$,  and either $W = X$ or 
$W$ has finite intersection with forward $\sigma$-orbits.  

Furthermore, if $(2)$ and $(1)$ hold, then $\sh{R}$ is a finite right module over $\sh{R}\ver{n}$, and there are a closed subscheme $W'$ of $W$, 
with ideal sheaf $\sh{K}'$, so that  $(W')^{\red} = W^{\red}$, and an integer $n_0$  such that 
\[\sh{R}(X, \Lsh_n, \sigma^n, Z)_{\geq n_0} = \sh{R}(X, \Lsh_n, \sigma^n, W')_{\geq n_0} = \sh{K}' (\sh{B}\ver{n})_{\geq n_0} .\]
  That is, any noetherian right idealizer is a finite right module over a right idealizer at a subscheme without fixed components.
\end{theorem}

Before proving Theorem~\ref{thm-rtnoeth}, we give some preliminary lemmas.  

\begin{lemma}\label{lem-BOX}
Let $X$ be a projective variety, let $\sigma \in \Aut_{\kk} X$, and let $\Lsh$ be an invertible sheaf on $X$.  Let $\sh{B}$ be a graded $(\struct_X, \sigma)$-sub-bimodule algebra of $ \sh{B}(X, \Lsh, \sigma)$, and let $\sh{R}$ and $\sh{R}'$ be graded $( \struct_X, \sigma)$-sub-bimodule algebras of $\sh{B}$.  Suppose that $\sh{R}$ is right noetherian and contains a nonzero graded right ideal of $\sh{B}$ and that there is some $n_0$ so that 
\[ \sh{R}_{\geq n_0} \subseteq \sh{R}'_{\geq n_0}.\]
  Then $\sh{R}'$ is right noetherian.  If $\sh{R} \subseteq \sh{R}'$, then $\sh{R}'$ is a coherent right $\sh{R}$-module.  
\end{lemma}
\begin{proof}
By Lemma~\ref{lem-catequiv}, without loss of generality we may assume that $\Lsh = \struct_X$.  

We note that $\sh{R}_{\geq n_0}$ also contains a nonzero graded  right ideal of $\sh{B}$.  Further,  $\sh{R} \cap \sh{R}'$ is also right noetherian, as $(\sh{R} \cap \sh{R}')_{\geq n_0} = \sh{R}_{\geq n_0}$.  Thus without loss of generality we may assume that $\sh{R} \subseteq \sh{R}'$.

Let $\sh{J}$ be a nonzero graded right ideal of $\sh{B}$ that is contained in $\sh{R}$; let $m$ be such that $\sh{J}_m \neq 0$.  Since $X$ is a variety, there is  an invertible ideal sheaf $ \sh{H}$ contained in $\sh{J}_m$.    As $\sh{R}$ is right noetherian and $\sh{H} \cdot \sh{R}' \subseteq \sh{J}_m \cdot \sh{B} \subseteq \sh{R}$, we see that $ \sh{H}\cdot \sh{R}'$ is a coherent right $\sh{R}$-module.  Lemma~\ref{lem-catequiv} now implies that $\sh{R}'$ is a coherent right $\sh{R}$-module.

 Any right ideal of $\sh{R}'$ is also a right $\sh{R}$-submodule, and so is coherent as an $\sh{R}$-module.  It is thus also coherent as an $\sh{R}'$-module.  Thus $\sh{R}'$ is right noetherian.
\end{proof}

The next few lemmas involve translating general results on idealizers  to the context of  bimodule algebras.   The proofs are all easy generalizations of the original proofs.

The following result is originally due to Robson \cite[Proposition~2.3(i)]{Robson1972}, although we will follow Stafford's restatement of it.
\begin{lemma}\label{lem-Rob}
{\em (\cite[Lemma~1.1]{Stafford1985})}
Let $I$ be a right ideal of a right noetherian ring $B$, and let $R = \I_B(I)$.  If  $B/I$ is a right noetherian $R$-module, then $R$ is right noetherian. \qed
\end{lemma}

Our version of this is the following lemma.

\begin{lemma}\label{lem-sheafRob}
Let $X$ be a projective variety, let $\sigma \in \Aut_{\kk} X$, and let $\Lsh$ be an invertible sheaf on $X$.  Let $\sh{B} = \sh{B}(X, \Lsh, \sigma)$, and let $\sh{I} = \bigoplus (\sh{I}_n)_{\sigma^n}$ be  a nonzero graded right ideal of $\sh{B}$.  Let $\sh{R} = \I_{\sh{B}}(\sh{I})$.  Then  $\sh{B}/\sh{I}$ is a noetherian right  $\sh{R}$-module if and only if  $\sh{R}$ is right noetherian.
\end{lemma}
\begin{proof}
The proof is a straightforward translation of \cite[Proposition~2.3(i)]{Robson1972} into sheaf terminology.  By Lemma~\ref{lem-catequiv}, without loss of generality we may let $\Lsh = \struct_X$.    Thus all $\sh{R}_n$ and all $\sh{I}_n$ are ideal sheaves on $X$.

 By Lemma~\ref{lem-BOX}, if  $\mathcal{R}$ is right noetherian, certainly $\sh{B}_{\sh{R}}$ and thus  $(\mathcal{B}/\sh{I})_{\mathcal R}$ are also.
So suppose that $\mathcal{B} / \mathcal{I} $ is a noetherian right $\mathcal{R}$-module.  Let $\mathcal{J}$ be a right ideal of $\mathcal{R}$; we will show that $\mathcal{J}$ is coherent.  Because $\mathcal{B}$ is right noetherian, we may choose a coherent sheaf $\mathcal{J}' \subseteq \mathcal{J}$ such that $\mathcal{J}' \cdot \mathcal{B} = \mathcal{J}\cdot \sh{B}$.  It suffices to show that 
$\mathcal{J} / (\mathcal{J}' \cdot  \sh{R})$ is a coherent right $\sh{R}$-module. 

Now,  $\mathcal{J} /( \mathcal{J}' \cdot \sh{R})$ is a submodule of $((\mathcal{J}'\cdot \sh{B}) \cap \mathcal{R}) / (\mathcal{J}'\cdot \sh{R})$.  Further, it is killed by $\mathcal{I}$ and so is a subfactor of $\mathcal{J}' \otimes (\mathcal{B} / \mathcal{I})$.  Since $\sh{B}/\sh{I}$ is a noetherian right $\sh{R}$-module, so is  $\mathcal{J}' \otimes (\mathcal{B} / \mathcal{I})$.  Thus the subfactor $\sh{J}/(\sh{J}' \cdot \sh{R})$ is coherent.  
\end{proof}

The criterion in Lemma~\ref{lem-Rob} can be hard to test.    Stafford \cite[Lemma~1.2]{Stafford1985} gave a different criterion for an idealizer to be noetherian; it was later slightly strengthened by Rogalski \cite[Proposition~2.1]{R-idealizer}.   We give the following version, which is adequate for our needs.  

\begin{lemma}\label{lem-Rog-I}
Let $B$ be a right noetherian domain, let $I$ be a right ideal of $B$, and let $R  = \I_B(I)$.  Then 
the following are equivalent:

$(1)$ $R$ is right noetherian;

$(2)$ $B_R$ is finitely generated, and for all right ideals $J$ of $B$ such that $J \supseteq I$, we have that $\Hom_B(B/I, B/J)$ is a noetherian right $R$-module (or $R/I$-module).
\end{lemma}
\begin{proof}
(2) $\Rightarrow$ (1) is \cite[Lemma~1.2]{Stafford1985}.
For (1) $\Rightarrow$ (2), note that if $R$ is noetherian, as $B$ is a domain we have $B_R \hookrightarrow R_R$ and so $B_R$ is finitely generated.  The rest of the argument is \cite[Proposition~2.1]{R-idealizer}.
\end{proof}

Our version of this is the following lemma:
\begin{lemma}\label{lem-shSt}
Let $X$ be a projective variety, and let $\sigma \in \Aut_{\kk} X$.  Let $\sh{B}$ be  a right noetherian graded $(\struct_X, \sigma)$-sub-bimodule algebra of the twisted bimodule algebra 
$\sh{B}(X, \struct_X, \sigma),$
 and let $\sh{I} = \bigoplus (\sh{I}_n)_{\sigma^n} $ be  a nonzero graded right ideal of $\sh{B}$.  Let $\sh{R} = \I_{\sh{B}}(\sh{I})$.  Suppose that for all graded right ideals $ \sh{J} \supseteq \sh{I}$ 
of $\sh{B}$,  for $n \gg 0$, 
\[ \sh{B}_n \cap \bigcap_{m \geq 0} (\sh{J}_{n+m}: \sh{I}_m^{\sigma^n}) = \sh{J}_n.\]
Then $\sh{R}$ is right noetherian.
\end{lemma}

\begin{proof}
We follow Stafford's proof of \cite[Lemma~1.2]{Stafford1985}.  
Assume that the hypotheses of the lemma hold; we claim that $\sh{B}/\sh{I}$ is a noetherian right $\sh{R}$-module.

Let $\mathcal{G}$ be a graded right $\sh{R}$-module with $\mathcal{I} \subseteq \mathcal{G} \subseteq \mathcal{B}$.  We seek to prove that $\mathcal{G}/\sh{I}$ is coherent.  Let $\mathcal{J}$ be the largest graded right ideal of $\mathcal{B}$ of the form $\mathcal{G}'\cdot \sh{I}$ for some coherent graded $\struct_X$-submodule $\mathcal{G}' $ of $\mathcal{G}$.   ($\mathcal{J}$ exists because $\sh{B}$ is right  noetherian.)  By maximality of $\mathcal{J}$, we have  $\mathcal{I} \subseteq \mathcal{J}$.  

Using Zorn's lemma, let $\mathcal{C}$ be the maximal quasicoherent subsheaf of $\mathcal{B}$ such that $\mathcal{C} \cdot \sh{I} \subseteq \mathcal{J}$.    Obviously, $\sh{C}$ is graded.  Note that 
\[ \sh{C}_n = \sh{B}_n \cap \bigcap_{m \geq 0} (\sh{J}_{n+m}: \sh{I}_m^{\sigma^n}).\]
Since   $\sh{C} \cdot \sh{R} \cdot \sh{I} \subseteq \sh{C} \cdot \sh{I} \subseteq \sh{J}$, we have that 
$\sh{C} \cdot \sh{R} \subseteq \sh{C}$ and $\sh{C}$ is a right $\sh{R}$-submodule of $\sh{B}$.  Since by assumption $\sh{C}_n = \sh{J}_n$ for $n \gg 0$, the right $\sh{R}$-module $\sh{C}/\sh{J}$ is in fact a coherent $\struct_X$-module.

We claim that $\sh{G} \subseteq \sh{C}$.  Suppose not.   We may choose a coherent graded $\struct_X$-submodule $\mathcal{G}''$ of $\mathcal{G}$ such that $\mathcal{G}'' \not\subseteq \mathcal{C}$, and so $\mathcal{G}'' \cdot \sh{I} \not\subseteq \mathcal{J}$.  Then $(\mathcal{G}' + \mathcal{G}'') \cdot \mathcal{I}  \supsetneqq \mathcal{J}$ by choice of $\mathcal{G}''$, contradicting the maximality of $\mathcal{J}$.  
Thus $\mathcal{G} \subseteq \mathcal{C}$. 
  
 Since $\sh{C}/\sh{J}$ is  a coherent $\struct_X$-module, so is the submodule $\mathcal{G}/\mathcal{J}$.  Since $\sh{J}_\sh{R}$ is coherent  and $\sh{G}/\sh{J}$ is a coherent $\struct_X$-module, $\sh{G}_{\sh{R}}$ is coherent.  Thus  $(\sh{G}/\sh{I})_{\sh{R}}$ is also coherent.   
Since $\sh{G}$ was arbitrary, we have shown that $\sh{B}/\sh{I}$ is a right noetherian $\sh{R}$-module as claimed.  

Applying Lemma~\ref{lem-sheafRob}, we obtain that $\sh{R}$ is a right noetherian bimodule algebra.
\end{proof}

We are almost ready to give the proof of Theorem~\ref{thm-rtnoeth}.   One technical difficulty in  the proof    is that if $Z$ has multiple components, it may be difficult to compute $(\sh{I}: \sh{I}^{\sigma^n})$ and thus $\sh{R}$.  However, if $Z$ is irreducible, then  computing $\sh{R}$ is straightforward; we record this in the next lemma.

\begin{lemma}\label{lem-onecomp}
Assume Notation~\ref{notnot}.  Suppose in addition that $Z$ is irreducible and without embedded components.   If $Z^{\red}$ has infinite order under $\sigma$, then $\sh{R} = \struct_X \oplus \sh{I} \sh{B}_+$.  
\end{lemma}
\begin{proof}
 Let $\sh{P}$ be the ideal sheaf of $Z^{\red}$.  For $n \geq 1$, clearly $\sh{I}^{\sigma^n} \not\subseteq \sh{P}$, since $\sh{I}^{\sigma^n}$ is $\sh{P}^{\sigma^n}$-primary and $\sh{P}^{\sigma^n} \neq \sh{P}$.  The result follows from Lemma~\ref{lem-primary}(2) and the identification $\sh{R}_n = (\sh{I}: \sh{I}^{\sigma^n}) \Lsh_n$.
 \end{proof}

We can now show that Theorem~\ref{thm-rtnoeth} holds under some additional assumptions on $Z$.

\begin{lemma}\label{lem-JANIS}
Assume Notation~\ref{notnot}.  Let  
\[ \sh{I} = \sh{K}_1 \cap \cdots \cap \sh{K}_c\]
be a minimal primary decomposition of $\sh{I} = \sh{I}_Z$, where each $\sh{K}_i$ is $\sh{Q}_i$-primary for some prime ideal sheaf $\sh{Q}_i$. 
 For $i = 1 \ldots c$, let  $Z_i$ be the primary component of $Z$ corresponding to $\sh{K}_i$, and let 
\[ \sh{R}^i = \I_{\sh{B}} (\sh{K}_i \sh{B}_+) = \sh{R}(X, \Lsh, \sigma, Z_i).\]

Suppose that for all $1 \leq i, j \leq c$ the set 
\[ \{ m \geq 0 \st \sh{K}_j^{\sigma^m} \subseteq \sh{Q}_i\} \]
is finite.  (In particular, we assume that the $\sh{Q}_i$ are of infinite order under $\sigma$.)  Then $\sh{R}_m = \sh{I} \sh{L}_m$ for $m \gg 0 $.  Further, the following are equivalent:

$(1)$  $\sh{R}$ is right noetherian;

$(2)$  $\sh{R}^i$ is right noetherian for $i = 1 \ldots c$;

$(3)$ $Z$ has finite intersection with forward orbits;

$(4)$ if $\sh{J}$ is an ideal sheaf on $X$ such that $\sh{J} \supseteq \sh{I}$, then  $(\sh{J} : \sh{I}^{\sigma^m}) = \sh{J}$ for $m \gg 0 $;

$(5)$ the bimodule algebra
$ \struct_X \oplus \sh{IB}_+$
is right noetherian.
\end{lemma}
We note that the assumptions of the lemma are satisfied if $Z$ consists of one primary component such that $Z^{\red}$ is of infinite order under $\sigma$.

\begin{proof}
By Lemma~\ref{lem-catequiv}, we may without loss of generality assume that $\Lsh = \struct_X$.   By Lemma~\ref{lem-primary}(2)
\[ (\sh{I}: \sh{I}^{\sigma^m})  = \sh{I} \]
for $m \gg 0$.  Thus 
$\sh{R}_m = \sh{I}$ for $m \gg 0$, as claimed.  Note that this implies that (1) $\iff$ (5).

(1) $\Rightarrow$ (2).   Fix $1 \leq i \leq c$.   By Lemma~\ref{lem-onecomp}, $(\sh{R}^i)_m = \sh{K}_i$ for all $m \geq 1$.  As  $\sh{R}_m = \sh{I}$ for $m \gg 0$, there is some $m_0$ so that for $m \geq m_0$
\[ \sh{R}(X, \struct_X, \sigma, Z)_{m} = \sh{I}  \subseteq \sh{K}_i =  (\sh{R}^i)_{m}.\]
By Lemma~\ref{lem-BOX}, $\sh{R}^i$ is right noetherian.

(2) $\Rightarrow$ (3) Since $Z$ is the set-theoretic union of finitely many irreducible components, it is enough to prove (3) in the case that $\sh{I}$ is itself primary; that is, in the case that $i=1$.  In this case, since $\sh{R} = \sh{R}^1$ is noetherian by assumption, by Lemma~\ref{lem-BOX} $\sh{B}_{\sh{R}}$ is coherent.   By Lemma~\ref{lem-onecomp}, $\sh{R} = \struct_X \oplus \sh{I} \sh{B}_+$.  

Fix $x \in X$, and let $\sh{I}_x$ be its ideal sheaf.  Let 
\[ \sh{M} = \bigoplus_{n \geq 0} (\sh{I}_x: \sh{I}^{\sigma^n})_{\sigma^n} \subseteq \sh{B} = \bigoplus_{n \geq 0} (\struct_X)_{\sigma^n}.\]
Let $m \geq 1$ and $n \geq 0$.  
By Lemma~\ref{lem-onecomp},  $\sh{R}_m = \sh{I}$.  Therefore, 
\[( \sh{M}_n \cdot \sh{R})_{m+n}  = (\sh{I}_x: \sh{I}^{\sigma^n}) \sh{I}^{\sigma^n} \subseteq \sh{I}_x  \subseteq \sh{M}_{m+n},\]
 so $\sh{M}$ is a right $\sh{R}$-submodule of $\sh{B}$.  It is therefore coherent, and so is the quotient $\sh{M}/\sh{I}_x \sh{B}$.   Since $\sh{M} \cdot \sh{IB}_+ \subseteq \sh{I}_x \sh{B}$,  the $\sh{R}$-action on $\sh{M}/\sh{I}_x \sh{B}$ factors through $\sh{R}/\sh{IB_+} = \struct_X$.  In other words, $\sh{M}/\sh{I}_x\sh{B}$ is a noetherian and therefore coherent $\struct_X$-module, and  so the ideal sheaves $(\sh{I}_x: \sh{I}^{\sigma^n})$ and $\sh{I}_x$ are equal for $n \gg 0$.  For fixed $n$, this is true if and only if $x \not \in \sigma^{-n}Z$ or $\sigma^n(x) \not\in Z$.  Thus $\{n \geq 0 \st \sigma^n(x) \in Z\}$ is finite.

(3) $\Rightarrow$ (4).   Let $\sh{P}$ be a nonzero prime ideal sheaf, defining a reduced and irreducible subscheme $W \subset X$.   Since for any $m \in \ZZ$ we have that $\sh{I}^{\sigma^m} \subseteq \sh{P}$ if and only if $\sigma^m(W)$ is (set-theoretically) contained in $Z$, we see that the set 
\[ \{ m \geq 0 \st \sh{I}^{\sigma^m} \subseteq \sh{P} \}\]
is finite.

Now let $\sh{J} \supseteq \sh{I}$ be an ideal sheaf on $X$, and let $\sh{J} = \sh{J}_1 \cap \cdots \cap \sh{J}_e$ be a primary decomposition of $\sh{J}$, where $\sh{J}_i$ is $\sh{P}_i$-primary for a suitable prime ideal sheaf $\sh{P}_i$.   For  $m \gg 0$ and  for $i = 1 \ldots e$, we have $\sh{I}^{\sigma^m} \not\subseteq \sh{P}_i$.   
Therefore by Lemma~\ref{lem-primary}(2), $(\sh{J}: \sh{I}^{\sigma^m}) = \sh{J}$ for $m \gg 0$.  

(4) $\Rightarrow$ (1). 
 Suppose that for all $\Ish \subseteq \mathcal{J} \subseteq \struct_X$, we have  $(\sh{J}: \sh{I}^{\sigma^n}) = \sh{J}$ for $n \gg 0$.    Let $\sh{F} \supseteq \sh{IB}$ be a graded right ideal of $\sh{B}$, and for all $m \geq 0$ let
\[ \sh{C}_m =  \bigcap_{n \geq 0} (\sh{F}_{n+m}: \sh{I}^{\sigma^m}).\] 
 We saw in Example~\ref{eg-twist} that  the categories $\rqgr \sh{B}$ and $\struct_X \lmod$ are equivalent, and that   
there is an ideal sheaf $\mathcal{J} \subseteq \struct_X$ such that, for some $k$, we have 
\[\mathcal{F}_{\geq k} = \bigoplus_{m \geq k} ( \mathcal{J})_{\sigma^m}.\]
By construction,  $\sh{J} \supseteq \sh{I}$.   For $m \geq k$, we have $\sh{C}_m = (\sh{J}: \sh{I}^{\sigma^m})$.   By assumption, this is equal to $\sh{J} = \sh{F}_m$ for $m \gg k$, and so the hypotheses of Lemma~\ref{lem-shSt} hold.  By Lemma~\ref{lem-shSt},  $\sh{R}$ is right noetherian.
\end{proof}

We now prove Theorem~\ref{thm-rtnoeth}.

\begin{proof}[Proof of Theorem~\ref{thm-rtnoeth}]
We recall our notation:  
\[
\sh{I} = \sh{J}_1 \cap \cdots \cap \sh{J}_c \cap \sh{K}_1 \cap \cdots \cap \sh{K}_e
\]
is a minimal primary decomposition of $\sh{I}$, where each $\sh{J}_i$ is $\sh{P}_i$-primary for some  prime ideal sheaf $\sh{P}_i$ of finite order under $\sigma$, and each $\sh{K}_j$ is $\sh{Q}_j$-primary for some prime ideal sheaf $\sh{Q}_j$ of infinite order under $\sigma$.  The ideal sheaf $\sh{K} =  \sh{K}_1 \cap \cdots \cap \sh{K}_e$ defines the closed subscheme $W$ of $X$, and  $\sh{J} = \sh{J}_1 \cap \cdots \cap \sh{J}_c$.  

By Lemma~\ref{lem-catequiv}, we may without loss of generality assume that $\Lsh = \struct_X$.  

$(1) \Rightarrow(2)$.  
Suppose that $\sh{R}$ is right noetherian.   We first show this implies that there is some $n$ so that all $\sh{J}_i$ are fixed by $\sigma^n$.     Suppose, in contrast, that for some $i$ there is no $n$ with  $\sh{J}_i^{\sigma^n} = \sh{J}_i$.   Since Veronese rings of $\sh{R}$ are also  right noetherian  and $\sh{P}_i$ has finite order under $\sigma$, we may assume without loss of generality that $\sh{P}_i$ is fixed by $\sigma$.  

Let $m \geq 1$.  Since $(\sh{J}_i)^{\sigma^m} \neq \sh{J}_i$, by minimality of the primary decomposition \eqref{kiri2}, it is clear that  $\sh{I}^{\sigma^m} \not\subseteq \sh{J}_i$.  By Lemma~\ref{lem-prim2} $(\sh{J}_i: \sh{I}^{\sigma^m}) \neq \struct_X$ is $\sh{P}_i$-primary.  Therefore  
\[\sh{R}_m = (\sh{I}:\sh{I}^{\sigma^m}) \subseteq (\sh{J}_i: \sh{I}^{\sigma^m}) \subseteq \sh{P}_i\]
 for all $m \geq 1$.
 
 Let $\sh{B} = \sh{B}(X, \struct_X, \sigma)$.  For any $k \geq 0$, we have 
\[(\sh{B}_{\leq k} \cdot \sh{R})_{k+1} = \sum_{j=0}^k (\sh{R}_{k+1-j})^{\sigma^j} \subseteq \sh{P}_i \neq \struct_X = \sh{B}_{k+1}.\]
We see that $\sh{B}_{\sh{R}}$ is not finitely generated; by Lemma~\ref{lem-BOX},  this contradicts the assumption that $\sh{R}$ is right noetherian.   Thus $\sh{J}_i$ is of finite order under $\sigma$.

As this holds for all $i$, there is some $n$ so that $\sh{J}^{\sigma^n} = \sh{J}$.  Suppose that $W \neq X$.  
Since $W$ has finite intersection with  forward $\sigma$-orbits  if and only if  $W$ has finite intersection with forward $\sigma^n$-orbits, without loss of generality we may replace $R$ by  the Veronese $\sh{R}\ver{n}$ and assume that $\sh{J}$ is $\sigma$-invariant.  
Suppose that $W$ has infinite intersection with some forward $\sigma$-orbit.  We will derive a contradiction.  

For $i = 1 \ldots  e$, let $W_i$ be the primary component of $Z$ defined by $\sh{K}_i$, and let $Y_i = W_i^{\red}$ be the subvariety defined by  the prime ideal sheaf $\sh{Q}_i$.  We claim that there is some $i$ so that
\begin{itemize}
\item[(i)]  $ Y_i  \not\subseteq \sigma^{ -m}(W)$  for $m \geq 1$;
\item[(ii)] for some $x \in X$,  the set $\{m \geq 0 \st \sigma^m(x) \in Y_i\}$ is infinite.
\end{itemize}
To see this, note that we may define a strict partial order $\prec$ on the set of the $Y_i$ by defining 
\[ Y_i \prec Y_j \mbox{ if $Y_i \subseteq \sigma^{-m}(Y_j)$ for some $m \geq 1$.} \]
The order $\prec$ is strict because each $Y_i$ has infinite order under $\sigma$.  
Now, (ii) holds for some $Y_i$ by assumption.  Thus (ii) holds for some $Y_i$ that is  maximal under $\prec$.  But (i) holds for any such maximal  $Y_i$, as the ideal sheaf of $Y_i$ is prime.

Let $Y_i$ satisfy (i) and (ii).   We thus have $\sh{K}^{\sigma^m} \not\subseteq \sh{Q}_i$ for any $m \geq 1$.  
As 
\[ \sh{I}^{\sigma^m} = \sh{K}^{\sigma^m} \cap \sh{J}^{\sigma^m} = \sh{K}^{\sigma^m} \cap \sh{J},\]
by Lemma~\ref{lem-primary}(1) we have
\[ \sh{R}_m = (\sh{I}:\sh{I}^{\sigma^m}) \subseteq  (\sh{K}_i: \sh{J})\]
for all $m \geq 1$.   By minimality of the primary decomposition \eqref{kiri2} and Lemma~\ref{lem-prim2}, the ideal sheaf $(\sh{K}_i : \sh{J})$ is $\sh{Q}_i$-primary.  

Let $V$ be the closed subscheme of $X$ defined by $(\sh{K}_i : \sh{J})$.     By Lemma~\ref{lem-onecomp}, 
\[ \sh{R}(X, \struct_X, \sigma, V) = \struct_X \oplus (\sh{K}_i: \sh{J}) \sh{B}_+,\]
so 
\[ \sh{R}(X, \struct_X, \sigma, Z) \subseteq \sh{R}(X, \struct_X, \sigma, V).\]
 Thus by Lemma~\ref{lem-BOX}, $\sh{R}(X, \struct_X, \sigma, V)$ is right noetherian.  But  $V$ also  has infinite intersection with some forward $\sigma$-orbit. 
This is impossible, by Lemma~\ref{lem-JANIS}.    

Thus $W$ has finite intersection with forward $\sigma$-orbits.

$(2) \Rightarrow (1)$.  
Suppose that (2) holds.   We claim that 
\beq \label{MELISSA}
(\sh{I}:\sh{I}^{\sigma^m}) = \sh{R}_m = (\sh{I}: \sh{J}^{\sigma^m}) \quad \mbox{ for $m \gg 0$} .
\eeq  

  If $W = X$ then $\sh{I} =\sh{J}$ and \eqref{MELISSA} holds for all $m$.
If  $W \neq X$ has finite intersection with forward $\sigma$-orbits, then 
for $m \gg 0$, $\sh{K}^{\sigma^m}$ is not contained in any minimal prime over $\sh{I}$. Thus  by
Lemma~\ref{lem-primary}(1) we have that
\[ (\sh{I}: \sh{I}^{\sigma^m}) \subseteq (\sh{I}: \sh{J}^{\sigma^m})\]
for $m \gg 0$.  As the other containment is automatic, we see that \eqref{MELISSA} holds.

Now, if $n | m$ then 
\[(\sh{I}: \sh{J}^{\sigma^m}) = (\sh{I}: \sh{J}) = (\sh{K}: \sh{J})\]
and so \eqref{MELISSA} implies in particular that $\sh{R}\ver{n}$ and $\struct_X \oplus (\sh{K}: \sh{J}) (\sh{B}\ver{n})_+$ 
are equal in large degree.

If $W= X$ then  $(\sh{K}: \sh{J}) = \struct_X$ and $\sh{R}\ver{n} = \sh{B}\ver{n}$.  If $W$ has finite intersection with forward $\sigma$-orbits, then note that $(\sh{K}: \sh{J})$ is the intersection of the $\sh{Q}_i$-primary ideal sheaves $(\sh{K}_i: \sh{J})$.  Let $W'$ be the closed subscheme of $X$ defined by $\sh{K}' = (\sh{K}: \sh{J})$; then $W'$ also has finite intersection with forward $\sigma$-orbits, and $(W')^{\red} = W^{\red}$.  Applying Lemma~\ref{lem-JANIS} to $\sh{R}\ver{n}$, we obtain that $\sh{R}\ver{n}$ is right noetherian.  By Lemma~\ref{lem-BOX},  $\sh{B} \ver{n}$ is a coherent right $\sh{R}\ver{n}$-module.    

Thus in either case, $\sh{B} \ver{n}$ is a coherent right $\sh{R} \ver{n}$-module.  Therefore, for any $m$ the right ideal 
\[ (\sh{I}: \sh{J}^{\sigma^m})( \sh{B} \ver{n})\]
of $\sh{B} \ver{n}$ 
is a coherent right $\sh{R}\ver{n}$-module.  Applying \eqref{MELISSA} for $m = 0 \ldots n-1$, we obtain that 
$\sh{R}$ is a finitely generated right $\sh{R} \ver{n}$-module and so $\sh{R}$ is right noetherian.
\end{proof}

\begin{example}\label{eg-notnoeth}
We give an example illustrating what  goes wrong when $\sh{J}$ is of infinite order under $\sigma$.  Let $X = \PP^2$  and let 
\[ \sigma = \left[ \begin{array}{ccc}
		1 & & \\
		& p & \\
		& & q
		\end{array} \right] \]
for some $p, q \in \kk^{*}$ that are not roots of unity.  Let $B = B(X, \struct(1), \sigma)$.   It is easy to see that $B$ may be presented by generators and relations as:
\[ B \cong \kk \ang{X,Y,Z}/(XY-pYX, XZ-qZX, YZ-(qp^{-1})ZY).\]

Let $a = [0:0:1]$ and let $\struct = \struct_{X, a}$.  Let $\mf{m}$ be the maximal ideal of $\struct$.  As $\sigma(a)  = a$, the automorphism  $\sigma$ acts on  $\struct$ via 
\begin{align*}
\sigma(x) & = x \\
\sigma(y) & = py,
\end{align*}
where $(x,y)$ is an appropriate system of parameters for $\struct$. 

Let $\sh{I}$ be the ideal sheaf cosupported at $a$ so that $\sh{I}_a= (x+y, \mf{m}^2)$.     Then for any $n$ we have
\[ (\sh{I}^{\sigma^n})_{a} = (x+p^ny, \mf{m}^2).\] 
Let $\sh{M}$ be the ideal sheaf of $a$.  We leave to the reader the computation that 
\[ (\sh{I}: \sh{I}^{\sigma^n}) = \sh{M}.\]
Thus, if $Z$ is the subscheme defined by $\sh{I}$, we have that
\[ R(X, \struct(1), \sigma, Z) = \kk + xB + yB.\]
This ring is not finitely generated, and so not right or left noetherian.  Neither is the bimodule algebra $\sh{R}(X, \struct(1), \sigma, Z) = \struct_X \oplus \sh{MB}_+$.
\end{example}

\section{Left noetherian bimodule algebras}\label{LEFT-NOETH-BMA}

We now turn to considering when $\sh{R}(X, \Lsh, \sigma, Z)$ is left noetherian.  
Since one of our main goals is to understand when idealizers are noetherian, from now on we will assume the  condition 
\begin{assnot}\label{assnot-FOO}
Let $X$ be a projective variety, let $\sigma \in \Aut_{\kk} X$, and let $\Lsh$ be an invertible sheaf on $X$.  Let $Z$ be a closed subscheme of $X$ and let $\sh{I} = \sh{I}_Z$ be its defining ideal.  Let 
\[ \sh{B} = \sh{B}(X, \Lsh, \sigma)\]
and let
\[ \sh{R} = \sh{R}(X, \Lsh, \sigma, Z) = \bigoplus_{n \geq 0} \bigl( ( \sh{I}: \sh{I}^{\sigma^n}) \Lsh_n \bigr)_{\sigma^n}.\]

We assume that $(\sh{I}: \sh{I}^{\sigma^n}) = \sh{I}$ for $n \gg 0$; that is, that $\sh{R}_n = \sh{IL}_n$ for $n \gg 0$.
\end{assnot}

Assumption-Notation~\ref{assnot-FOO} is satisfied in the situations of most interest to us.  In particular, by Theorem~\ref{thm-rtnoeth}, any right noetherian idealizer bimodule algebra is, up to a finite extension,  one whose defining data satisfies Assumption-Notation~\ref{assnot-FOO}.   Furthermore, if $Z$ is irreducible and $Z^{\red}$ is of infinite order under $\sigma$, then Lemma~\ref{lem-onecomp} implies that Assumption-Notation~\ref{assnot-FOO} is satisfied.  More generally, we have:

\begin{lemma}\label{lem-niceconds}
Assume Notation~\ref{notnot}.  Let $Z_1, \ldots, Z_c$ be the primary components of $Z$.  For $i = 1 \ldots c$, let $Y_i = Z_i^{\red}$.   
If  $\{ n \geq 0 \st \sigma^n(Y_i) \subseteq Z\}$ is finite for all $i$, then  $\sh{R}_n = \sh{IL}_n$ for $n \gg 0$, and so Assumption-Notation~\ref{assnot-FOO} is satisfied. 
\end{lemma}

\begin{proof}
This follows directly from Lemma~\ref{lem-primary}(2).
\end{proof}

As in the previous section, we will translate results on idealizer subrings to the context of bimodule algebras.   We quote  a result of Rogalski; we note that the original result was stated for left ideals of noetherian rings.  

\begin{proposition}\label{prop-R-Tor}
{\em(\cite[Proposition~2.2]{R-idealizer})}
If $R = \mathbb{I}_B(I)$ for some right ideal $I$ of a  noetherian ring $B$, then $R$ is left noetherian if and only if $R/I$ is a left noetherian ring and for all left ideals $J$ of $B$, the left $R$-module $\Tor_1^B(B/I, B/J)$ is noetherian. \qed
\end{proposition}

We  note that if  $R/I$ is finite-dimensional,  this result reduces to saying that $R$ is left noetherian if and only if $\Tor_1^B(B/I, B/J)$ is a finite-dimensional vector space for all left ideals $J$ of $B$.

We now prove a version of Proposition~\ref{prop-R-Tor} for the bimodule algebra $\mathcal{R}$.   

\begin{proposition}\label{prop-ibma-lnoeth}
Assume Assumption-Notation~\ref{assnot-FOO}.  Then $\sh{R}$ is left noetherian if and only if for all graded left ideals $\sh{H}$ of $\sh{B}$, we have 
\[ \sh{I} \cap \sh{H}_n = (\sh{I} \sh{H})_n\]
for $n \gg 0$.
\end{proposition}
\begin{proof}
We follow Rogalski's proof of Proposition~\ref{prop-R-Tor}.  

Since 
 $(\mathcal{IB} \cap \mathcal{H}) / \mathcal{IH}$ is a subfactor of $_{\mathcal{R}} {\mathcal R}$ that is killed on the left by $\mathcal{IB}$, if $\sh{R}$ is  left noetherian then this is a coherent module over $\sh{R}/\sh{IB}$ and so is certainly  a coherent $\struct_X$-module.   
 
 For the other direction, suppose that for all graded left ideals $\sh{H}$ of $\sh{B}$ we have 
 \[ \sh{I} \cap \sh{H}_n = (\sh{IH})_n\]
 for $n \gg 0$.  
 Let  $\mathcal{K}$ be a graded  left ideal of $\mathcal{R}$.   Since $\sh{B}$ is noetherian, we may choose a graded coherent $\struct_X$-submodule $\mathcal{K}'$ of $\mathcal{K}$ such that 
$\mathcal{B} \cdot \sh{K} = \mathcal{B} \cdot \sh{ K}'$.  Since $\mathcal{K} /( \mathcal{R} \cdot \sh{K}')$ is a submodule of $((\mathcal{B} \cdot \sh{K}') \cap \sh{R})/(\mathcal{R} \cdot \sh{K}')$, it is enough to show  that $((\sh{B} \cdot \mathcal{K}) \cap \sh{R}) / \mathcal{K}$ is a noetherian left $\sh{R}$-module for any {\em coherent} graded left ideal $\mathcal{K}$ of $\sh{R}$.

But now consider the exact sequences of left $\sh{R}$-modules 
\beq\label{esa}
0 \to \frac{\mathcal{K}}{\mathcal{IB} \cdot \sh{K}} \to \frac{(\mathcal{B} \cdot \sh{K} )\cap \sh{R}}{\mathcal{IB} \cdot \sh{K}} \to \frac{(\mathcal{B} \cdot \sh{K} )\cap \sh{R}}{\mathcal{K}} \to 0
\eeq
and
\beq\label{esb}
0 \to \frac{(\mathcal{B} \cdot \sh{K}) \cap \mathcal{IB}}{\mathcal{IB} \cdot \sh{K}} \to \frac{(\mathcal{B} \cdot \sh{K}) \cap \sh{R}}{\mathcal{IB} \cdot \sh{K}} \to \frac{(\mathcal{B} \cdot \sh{K}) \cap \sh{R}}{(\mathcal{B} \cdot \sh{K}) \cap \mathcal{IB}} \to 0.
\eeq
Since  
$\fracc{((\mathcal{B} \cdot \sh{K}) \cap \sh{R})}{((\mathcal{B} \cdot \sh{K}) \cap \mathcal{IB})}$ is a  coherent $\struct_X$-module, 
we see that $((\sh{B} \cdot \mathcal{K}) \cap \sh{R}) / \mathcal{K}$ is noetherian if 
$\fracc{((\mathcal{B} \cdot \sh{K}) \cap \mathcal{IB})}{\mathcal{IB} \cdot \sh{K}} $ is noetherian.  
Since $\mathcal{B} \cdot \sh{K}$ is a left ideal of $\sh{B}$, by assumption 
$\fracc{((\mathcal{B} \cdot \sh{K}) \cap \mathcal{IB})}{\mathcal{IB} \cdot \sh{K}} $
  is a coherent $\struct_X$-module.    In particular, it is noetherian.  Thus $\sh{R}$ is left noetherian.
\end{proof}

This allows us to give a necessary and sufficient  condition for $\sh{R}$ to be left noetherian:
\begin{proposition}\label{prop-Tor-noeth}
Assume  Assumption-Notation~\ref{assnot-FOO}.    Then  $\mathcal{R} = \sh{R}(X, \Lsh, \sigma, Z)$ is left noetherian if and only if 
for all closed subschemes $Y \subseteq X$ the set 
\[ \{ n \geq 0 \st \shTor_1^X(\struct_{\sigma^n Z}, \struct_Y) \neq 0 \}\]
 is finite.
\end{proposition}

\begin{proof}
 Let $\sh{J}$ be an ideal sheaf defining a closed  subscheme $Y$ of $X$.  
 There are identifications of $\struct_X$-modules
 \begin{multline*}
  \frac{\sh{IB} \cap (\sh{B} \cdot \sh{J})}{\sh{IB} \cdot \sh{J}} \cong \bigoplus_{n \geq 0} \frac{\sh{I} \cap \sh{J}^{\sigma^n}}{\sh{IJ}^{\sigma^n}} \otimes \Lsh_n \\
  \cong \bigoplus_{n \geq 0} \shTor^X_1 (\struct_Z, \struct_{\sigma^{-n} Y}) \otimes \Lsh_n \cong \bigoplus_{n \geq 0} \shTor^X_1(\struct_{\sigma^n Z}, \struct_Y) \otimes \Lsh_n,
  \end{multline*}
 using 
  \cite[Exercise~3.1.3]{Weibel} and the local property of $\shTor$.  As $\sh{R}/\sh{IB}$ is a coherent $\struct_X$-module,  
\[ \frac{\sh{IB} \cap (\sh{B} \cdot \sh{J})}{\sh{IB} \cdot \sh{J}} \]
 is a coherent left $\sh{R}$-module if and only if it is a coherent $\struct_X$-module.  This is true if and only if the set $\{n \geq 0 \st \shTor^X_1(\struct_{\sigma^nZ}, \struct_Y) \neq 0 \}$ is finite.
 \end{proof}

We will explore the geometric meaning of the condition from Proposition~\ref{prop-Tor-noeth} in the next section.  For now, we give it a name for easy reference.

\begin{defn}
Let $X$ be a variety, let $\sigma \in \Aut_{\kk} X$, and let $Z$ be a closed subscheme of $X$.  We will (provisionally) say that the pair $(Z, \sigma)$ {\em has property $T$} if
for all closed subschemes $Y \subseteq X$ the set 
\[ \{ n \geq 0 \st \shTor_1^X(\struct_{\sigma^n Z}, \struct_Y) \neq 0 \}\]
 is finite.
\end{defn}
 
 \section{Critical transversality}\label{CT}

Assume Assumption-Notation~\ref{assnot-FOO}.  In the last two sections, we found necessary and sufficient conditions for $\sh{R} = \sh{R}(X, \Lsh, \sigma, Z)$ to be left or right noetherian.   We remark that there is a significant contrast between the two sides.  
 By Theorem~\ref{thm-rtnoeth}, for $\sh{R}$ to be right noetherian depends only on a mild condition on the orbits of points in $Z$.  In contrast, property $T$  that by Proposition~\ref{prop-Tor-noeth} determines when $\sh{R}$ is left noetherian is, a priori, much less transparent.     In this section, we show that property $T$ does have a natural geometric interpretation.     As remarked in the introduction,  it is an analogue of critical density and can be viewed as a transversality property; further, in many settings property $T$ may be interpreted as saying that $\sigma$ and $Z$ are in general position. 
 We will see later that in the presence of property $T$, it is possible to deduce many further nice properties of $\sh{R}$ and of $Z$.  

To begin, we define an algebraic generalization of classical transversality.

\begin{defn} \label{def-idealHT}\index{homologically transverse}
Let $X$ be a variety, and let $Y$ and $Z$ be closed subschemes of $X$.  We say that $Y$ and $Z$ are {\em homologically transverse} if
\[ \shTor^X_i(\struct_Z, \struct_Y) = 0\]
for all $i \geq 1$.
\end{defn}

While this appears as an arcane algebraic condition, it does in fact have a geometric basis.  As discussed in the introduction,  Serre \cite{Serre-LA} defines the  intersection multiplicity of two closed subschemes $Y$ and $Z$ of $X$ along the proper  component $P$ of their intersection by

\beq\label{intprod}
 i(Y, Z; P) = \sum_{i \geq 0}  (-1)^i \len_P (\shTor_i^X (\struct_Y, \struct_Z)).
\eeq

The higher $\shTor$ sheaves  are needed to correct for possible mis-counting from the na\"ive attempt to define $i(Y, Z; P)$ as $\len_P (\struct_Y \otimes \struct_Z)$.  
We may think of the non-vanishing of $\shTor_{\geq 1}^X(\struct_Y, \struct_Z)$ as indicating that $Y$ and $Z$ have an extremely non-transverse intersection (for example, the codimension of the intersection is smaller than $\codim Y + \codim Z$).

\begin{defn} \label{def-CT}\index{critically transverse}
Let $X$ be a variety, let $\sigma \in \Aut_{\kk} X$, and let $Z$ be a closed subscheme of $X$.  
Let $A \subseteq \ZZ$ be infinite.  We say that the set $\{ \sigma^n(Z) \}_{n \in A}$ is {\em critically transverse} if for all  closed subschemes $Y$ of $X$, $\sigma^n(Z)$ and $Y$ are homologically transverse for all but finitely many $n \in A$.  
\end{defn}

Critical transversality of $\{ \sigma^n Z\}_{n \in A}$ is a generic transversality property:   for any closed subscheme $Y$, it implies that the general translate of $Z$ is homologically transverse to $Y$.
In this section, we investigate critical transversality.  We will see in particular that critical transversality and property $T$ are equivalent.

 Recall that if $\sh{F}$ is a coherent sheaf on a projective variety $X$, we write $\hd_X(\sh{F})$ for the minimal length of a locally free resolution of $\sh{F}$. By \cite[Ex. III.6.3]{Ha}, we have
\[ \hd_X(\sh{F}) = \sup_{x \in X} \{ \pd_{\struct_{X,x}} \sh{F}_x \}.\]

The following lemma is due to Mel Hochster, and we thank him for allowing us to include it here. 
\begin{lemma}\label{lem-Mel}
{\em (Hochster)}
Suppose that $Z$ is homologically transverse to all parts of the singular stratification of $X$.  Then
\[
\hd_X( \struct_Z) \leq \dim X.
\]
\end{lemma}
\begin{proof}
Let $X = X^{(0)} \supset X^{(1)} \cdots \supset X^{(k)}$ be the singular stratification of $X$.  By assumption, $Z$ is homologically transverse to all $X^{(i)}$.  
By \cite[Corollary~19.5]{Eis}, 
\beq\label{Eis}
\hd_X (\struct_Z) = \sup \{j \st  \mbox{for some closed point } x \in X, \shTor^X_j(\struct_{Z}, \kk_x) \neq 0 \}.
\eeq
  So fix $x \in X$, and let $\struct = \struct_{X,x}$.   Let $F = \struct_{Z, x}$, considered as an $\struct$-module.  
  Let $i$ be such that $x \in X^{(i)} \smallsetminus X^{(i+1)}$.  Let  $J$ be the ideal of $X^{(i)}$ in $\struct$.  By assumption on $i$, $\struct/J$ is a regular local ring; in particular, $\pd_{\struct/J} \kk_x = \dim X^{(i)} \leq \dim X$. 
    
  The change of rings theorem for $\Tor$ \cite[Theorem~5.6.6]{Weibel} gives  a spectral sequence
  \beq\label{ss99}
  \Tor^{\struct/J}_p(\Tor^\struct_q(F, \struct/J), \kk_x) \Rightarrow \Tor^\struct_{p+q} (F, \kk_x).
  \eeq 
  Now by assumption, $Z$ is homologically transverse to $X^{(i)}$, and so \eqref{ss99} collapses for $q \neq 0$.  We obtain
  \[  \Tor^{\struct/J}_p(F \otimes_\struct (\struct/J), \kk_x) \cong \Tor^\struct_p(F, \kk_x).\]
  As $\struct/J$ is a regular local ring of dimension no greater than $\dim X$, we have that $\pd_{\struct/J} \kk_x \leq \dim X$ and so $\Tor^\struct_p(F, \kk_x) = 0$ if $p > \dim X$.    
  By \eqref{Eis},  $\hd_X(\struct_Z) \leq \dim X$.   
  \end{proof}

We now show that critical transversality and property $T$ are equivalent.  
\begin{lemma}\label{lem-Tor-reduction}
Let $A \subseteq \ZZ$.  The following are equivalent:

$(1)$ For all closed subschemes $Y$ of  $X$, the set 
\[\{ n \in A \st \shTor_1^X(\struct_{\sigma^n Z}, \struct_Y) \neq 0 \}\]
 is finite.
 
 $(2)$ For all reduced and irreducible closed subschemes $Y$ of  $X$, the set 
\[\{ n \in A \st \shTor_1^X(\struct_{\sigma^n Z}, \struct_Y) \neq 0 \}\]
 is finite.

$(3)$ For all  closed subschemes $Y$ of $ X$,  the set 
\[
A'(Y) = \{ n \in A \st \mbox{$\sigma^n Z$ is not homologically transverse to $Y$} \}
\]
 is finite.

In particular, $\{ \sigma^n(Z) \}_{n \geq 0}$ is critically transverse if and only if $(Z, \sigma)$ has property $T$.
\end{lemma}
\begin{proof}
The implications (3) $\Rightarrow$ (1) $\Rightarrow$ (2) are  trivial.  We prove (2) $\Rightarrow$ (3).  

Assume (2).  We may assume that $A$ is infinite.  We first claim that for any coherent sheaf $\sh{F}$ and for any $j \geq 1$, the set
\[ \{ n \in A \st \shTor^X_j(\struct_{\sigma^n Z}, \sh{F}) \neq 0 \} \]
is finite.    We induct on $j$.  As any coherent sheaf on a projective variety has a finite filtration by products of invertible sheaves with structure sheaves of reduced and irreducible closed subvarieties,  the claim is true for $j=1$.  Let $j > 1$ and fix a coherent sheaf $\sh{F}$.  Because $X$ is projective, it has enough locally frees, and  there is an exact sequence
\[ 0 \to \sh{K} \to \sh{L} \to \sh{F} \to 0\]
where $\Lsh$ is locally free and $\sh{K}$ is also coherent.  The long exact sequence in Tor implies that   \[ \shTor^X_j(\struct_{\sigma^n Z}, \sh{F}) \cong \shTor^X_{j-1}(\struct_{\sigma^n Z}, \sh{K})\]
for any $n$.  By induction, the right-hand side  vanishes for all but finitely many $n \in A$.  

The claim implies that $Z$ is homologically transverse to any $\sigma$-invariant closed subscheme of $X$, and, in particular, that $Z$ is homologically transverse to the singular stratification of $X$.  By Lemma~\ref{lem-Mel}, we have $\hd_X(\struct_Z) \leq \dim X$. Thus for a fixed $Y$, 
\[ 
A'(Y) = \{ n \in A \st \mbox{$\shTor^X_j(\struct_{\sigma^nZ}, \struct_Y) \neq 0$ for some $1 \leq j \leq \dim X$} \}.\]
By the claim, this is finite.
\end{proof}

\begin{corollary}\label{cor-CT-noeth}
Assume  Assumption-Notation~\ref{assnot-FOO}.  Then the bimodule algebra $\sh{R}$ is left noetherian if and only if $\{ \sigma^n Z\}_{n \geq 0}$ is critically transverse. 
\end{corollary}
\begin{proof}
Combine Lemma~\ref{lem-Tor-reduction} with Proposition~\ref{prop-Tor-noeth}.
\end{proof}

We next verify that critical transversality is, as claimed, a generalization of  critical density of the orbits of points.    We formally define:
\begin{defn}\label{def-CD}
Let $X$ be a variety, let $x \in X$, and let $\sigma \in \Aut_{\kk} X$.   Let $A \subseteq \ZZ$.  The set  $\{ \sigma^n(x) \st n \in A \}$ is {\em critically dense} if it is infinite and any infinite subset is dense in $X$.
\end{defn}

We first prove:

\begin{lemma}\label{lem-contain-Tor}
Let $W \subseteq V$ be closed subschemes of a scheme $X$.  
Then 
\[\shTor_1^X(\struct_V, \struct_W) \neq 0.\]
\end{lemma}

\begin{proof}
We work locally; let $W'$ be an irreducible component of $W$, and let $P = (W')^{\red}$.  Let  $\mf{m}$ be the maximal ideal of  the local ring $\struct = \struct_{X,P}$.  Let $J$ be the ideal of $\struct$ defining $V$ and let $I$ be the $\mf{m}$-primary ideal defining $W$ locally at $P$.  Then we have
\[\shTor_1^X(\struct_V, \struct_W)_P = \Tor^\struct_1(\struct/J, \struct/I) \cong  (J \cap I)/J I = J/JI,\]
as $J \subseteq I$.  By Nakayama's Lemma, this is nonzero.
\end{proof}

\begin{corollary}\label{cor-CT=CD}
Let $X$ be a variety, let $\sigma \in \Aut_{\kk}(X)$,  let $Z$ be a 0-dimensional closed subscheme of $X$, and let $A \subseteq \ZZ$ be infinite.  The following are equivalent:

$(1)$  
 $\{ \sigma^n(Z) \}_{n \in A}$ is critically transverse; 

$(2)$  $\{ \sigma^n(x)\}_{n \in A}$ is critically dense for all points $x \in \Supp Z$.
\end{corollary}

\begin{proof}
Because $\shTor_j^X (\struct_{\sigma^n Z} , \struct_Y)$ is supported on $\sigma^n Z \cap Y$ for any $j$, $(2) \Rightarrow (1)$.   We prove that $(1) \Rightarrow (2)$.  By working locally, we may assume that $Z$ is supported on a single point $x$.     Suppose that $(2)$ fails, so there is some infinite $A' \subseteq A$ and some reduced $W \subset X$ such that $\sigma^n(x) \in W$ for all $n \in A'$.  Then there is some, not necessarily reduced, $W'$ supported on $W$ such that $\sigma^n(Z) \subseteq W'$ for all $n \in A'$.  By Lemma~\ref{lem-contain-Tor},  $\shTor^X_1(\struct_{\sigma^n Z}, \struct_{W'}) \neq 0$ for any $n \in A'$.  Thus $(1)$ also fails.
\end{proof}

We now turn to investigating when critical transversality occurs.  As mentioned, critical transversality is a generic transversality property of the translates of $Z$, reminiscent of the Kleiman-Bertini theorem.  One naturally asks if there are simple conditions on $Z$ and $\sigma$ sufficient for $\{ \sigma^n Z \}$ to be critically transverse.  In particular, it is not immediately obvious, even working on nice varieties such as $\PP^n$, that critical transversality {\em ever} occurs, except for points with critically dense orbits.   

If instead of considering a single automorphism $\sigma$, we consider instead the action of an algebraic group $G$ on $X$, then there is a simple condition for generic transversality to hold for translates of $Z$.    This is
\begin{theorem}\label{thm-generalgroup}
{\em (\cite[Theorem~1.2]{S-KB})}
Let $X$ be a variety (defined over $\kk$) with a left action of  a smooth algebraic group $G$, and let $Z$ be a closed subscheme of  $X$.  Then the following are equivalent:
\begin{enumerate}
\item $Z$ is homologically transverse to all $G$-orbit closures in $X$.
\item For all closed subschemes $Y$ of $X$, there is a Zariski open and dense subset $U$ of $G$ such that for all closed points $g \in U$, the subscheme $g Z$  is homologically transverse to $Y$.
\end{enumerate}
\qed
\end{theorem}

(We note that if $G$ acts transitively, then $(1)$ and therefore $(2)$ are automatically satisfied; this case of Theorem~\ref{thm-generalgroup} was proved by Miller and Speyer \cite{MS}.)

In the remainder of this section, we apply Theorem~\ref{thm-generalgroup} to obtain a simple criterion for critical transversality, at least in characteristic 0.   It turns out that in many situations, critical transversality is, in a suitable sense, generic behavior.

We will use the following result of Cutkosky and Srinivas.  

\begin{theorem}\label{thm-CS}
{\em (\cite[Theorem~7]{CS1993})}
Let $G$ be a connected abelian algebraic group defined over an algebraically closed  field $\kk$ of characteristic 0.  Suppose that $g \in G$ is such that the cyclic subgroup $\ang{g}$ is dense in $G$.  Then any infinite subset of $\ang{g}$ is dense in $G$. \qed
\end{theorem}

Here is our simple condition for critical transversality:

\begin{theorem}\label{thm-whenCT}
Let $\kk$ be an algebraically closed field of characteristic 0, let $X$ be a variety over $\kk$, let $Z$ be a closed subscheme of $X$, and let $\sigma$ be an element of an algebraic group $G$ that acts on  $X$.  Let $A \subseteq \ZZ$ be infinite.    
  Then $\{ \sigma^n Z\}_{n \in A}$ is critically transverse if and only if $Z$ is homologically transverse to all reduced $\sigma$-invariant subschemes of $X$.
\end{theorem}

\begin{proof}
If $\{ \sigma^n Z\}_{n \in A}$ is critically transverse, then $Z$ is obviously homologically transverse to $\sigma$-invariant subschemes.   We prove the converse.  
Assume that $Z$ is homologically transverse to reduced $\sigma$-invariant subschemes of $X$.   We consider the abelian subgroup
\[H = \bbar{ \ang{ \sigma^n}} \subseteq G\]
Now, the closures of $H$-orbits in $X$ are $\sigma$-invariant and reduced.  Thus, by assumption, $Z$ is homologically transverse to all $H$-orbit closures, and we may apply Theorem~\ref{thm-generalgroup}.  Fix a closed subscheme $Y$ of $X$.  By Theorem~\ref{thm-generalgroup}, there is a dense open $U \subseteq H$ such that if $g \in U$, then $gZ$ and $Y$ are homologically transverse.

  Let $H^o$ be the connected component of the identity in $H$, so the components of $H$ are $H^o = \sigma^c H^o, \sigma H^o, \ldots, \sigma^{c-1} H^o$ for some $c \geq 1$.   As $\ang{\sigma^c}$ is dense in $H^o$, it is critically dense by Theorem~\ref{thm-CS}.
  
 Fix $0 \leq j \leq c-1$.  The set 
 \[ U_j = \sigma^{-j} (U \cap \sigma^j H^o)\]
 is an open dense subset of $H^o$.  By critical density, the set
 \[ \{ m \st \sigma^{mc} \not\in U_j \} \]
 is finite.  Thus
 \[  \{ n \st \sigma^n \not \in U \}= \bigcup_{j=0}^{c-1} \{ n \st n \equiv j \pmod c \mbox{ and } \sigma^{n-j} \not \in U_j \}\]
 is also finite.  That is to say,  for all but finitely many $n$ we have $\sigma^n \in U$ and $\sigma^n Z$ is homologically transverse to $Y$.  As $Y$ was arbitrary, $\{ \sigma^n Z \}_{n \in \ZZ}$ is critically transverse.  Thus $\{\sigma^n Z\}_{n \in A}$ is critically transverse.
\end{proof}

We note that the case of Theorem~\ref{thm-whenCT} where $Z$ is a point is proved in \cite[Theorem~11.2]{KRS}.

It is reasonable to say that $\sigma$ and $Z$ are in {\em general position} if $Z$ is homologically transverse to all $\sigma$-fixed subschemes of $X$.  Theorem~\ref{thm-whenCT} suggests the following conjecture:
\begin{conjecture}\label{conj-CT}
Let $\kk$ be an algebraically closed field of characteristic 0, and let $X/\kk$ be a projective variety.  Let $\sigma \in \Aut_{\kk} X$ and let $Z \subseteq X$ be a closed subvariety.  Then $\{ \sigma^n Z\}$ is critically transverse if and only if $\sigma$ and $Z$ are in general position.  
\end{conjecture}
If $Z$ is 0-dimensional, then this conjecture reduces to Bell, Ghioca, and Tucker's recent result \cite[Theorem~5.1]{BGT2008} that in characteristic 0, the orbit of a point under an automorphism is dense exactly when it is critically dense.  If $\sigma$ is an element of an algebraic group that acts on $X$, the conjecture is Theorem~\ref{thm-whenCT}.  In positive characteristic, the conjecture is known to be false; see \cite[Example~12.9]{R-generic} for an example of an automorphism $\sigma \in \PGL_n$ in positive characteristic with a dense but not critically dense orbit.

Suppose now that $\kk$ is uncountable (and algebraically closed) and that $X$ is a variety  over $\kk$.  We say that $x \in X$ is {\em very general} if there are proper subvarieties 
$\{ Y_i \st i \in \ZZ \}$
so that 
\[x \not \in \bigcup_i Y_i.\]

We can now show that critically transverse subschemes abound.  To make the statement easier, the following result is phrased in terms of $\PP^d$; a similar result holds for any variety with a transitive action by a reductive algebraic group.  

\begin{corollary}\label{cor-projspace}
Assume that $\kk$ is uncountable  and  $\chrr \kk = 0$.  Let $Z$ be a subscheme of $\PP^d$, and let $\sh{X}$ be the $\PGL_{d+1}$-orbit of $Z$ in the Hilbert scheme of $\PP^d$.  Let $\sh{Y} = \PGL_{d+1} \times \sh{X}$.  If $(\sigma, Z')$ is a very general element of $\sh{Y}$, then $\sigma$ and $Z'$ are in general position and the set $\{ \sigma^n Z' \}$ is critically transverse.  
\end{corollary}
\begin{proof}
By avoiding a countable union of proper subvarieties of $\PGL_{d+1}$, we may ensure that  the  eigenvalues of $\sigma$ are distinct and algebraically independent over $\QQ$.  This implies that the Zariski closure of $\{ \sigma^n \}$ in $\PGL_{d+1}$ is a torus $\mathbb{T}^d$, and that the only reduced subschemes fixed by $\sigma$ are unions of  coordinate linear subspaces with respect to an appropriate choice of coordinates.  There are finitely many of these; by  repeated applications of Theorem~\ref{thm-generalgroup} with $G = \PGL_{d+1}$ (or by \cite{MS}) we see that there is a dense open $U \subseteq \PGL_{d+1}$ such that for all $\tau \in U$, the subscheme $Z' = \tau Z$ is homologically transverse to all unions of coordinate linear subspaces; that is, $\sigma$ and $Z'$ are in general position.  By Theorem~\ref{thm-whenCT}, the set $\{ \sigma^n Z'\}$ is critically transverse.
\end{proof}

\section{Ampleness}\label{AMPLE}
We now return to noncommutative algebra.  
Our ultimate goal is to study, not the bimodule algebra $\sh{R} = \sh{R}(X, \Lsh, \sigma, Z)$, but the ring $R = R(X, \Lsh, \sigma, Z)$.  
In order to apply our knowledge of $\sh{R}$ to the ring $R$, we will need to control the ampleness, in the sense of Definition~\ref{def-ample}, of the sequence $\{ (\sh{R}_n)_{\sigma^n}\}$.    In this section we show that critical transversality of $\{\sigma^n Z\}$, together with $\sigma$-ampleness of $\Lsh$, is enough to show sufficient ampleness of the graded pieces of $\sh{R}$.

Throughout this section we assume Assumption-Notation~\ref{assnot-FOO}.  Thus to prove that the sequence $\{(\sh{R}_n)_{\sigma^n}\}$ is left or right ample, it suffices to prove that $\{ (\sh{I} \otimes \Lsh_n)_{\sigma^n}\}$ is left or right ample.

 Given $\sigma$-ampleness of $\Lsh$, right ampleness of $\{( \sh{R}_n)_{\sigma^n}\}$ is almost trivial; we record this in the next lemma.

\begin{lemma}\label{lem-rtample}
Assume  Assumption-Notation~\ref{assnot-FOO}.  Assume in addition that $\Lsh$ is $\sigma$-ample.  Then $\{ (\sh{R}_n)_{\sigma^n} \} $ is right ample.
\end{lemma}
\begin{proof}
From Assumption-Notation~\ref{assnot-FOO}, we know that  $\sh{R}_n = \sh{IL}_n = \sh{I} \otimes \Lsh_n$ for $n \gg 0$.  Fix a coherent sheaf $\sh{F}$.  Then for $n \gg 0$, we have  $\sh{F} \otimes \sh{R}_n = \sh{F} \otimes \sh{I} \otimes \Lsh_n$.  By $\sigma$-ampleness of $\Lsh$,  for $n \gg 0$ this is globally generated and has no higher cohomology.
\end{proof}

Left ampleness, however, is more subtle.  In fact,  we do not know when, in general, $\{(\sh{R}_n)_{\sigma^n}\}$ is left ample.  However, we will see that this does hold when $\sh{R}$ is left noetherian.

\begin{proposition}\label{prop-left-ample}
If $\Lsh$ is $\sigma$-ample and $\{ \sigma^n(Z)\}_{n \geq 0}$ is critically transverse,
then $\{ (\sh{I} \otimes \Lsh_n)_{\sigma^n} \}$ is a left  ample sequence.
\end{proposition}

We first prove:  

\begin{lemma}\label{lem-biample}
Let $\Lsh$ be a $\sigma$-ample invertible sheaf.

$(1)$ If $\sh{M}$ and $\sh{N}$ are coherent sheaves on $X$, then there is an integer $n_0$ so  $\mathcal{M} \otimes \Lsh_n \otimes \mathcal{N}^{\sigma^n}$ is globally generated for all $n \geq n_0$.

$(2)$ If  $\sh{E}$ and $\sh{F}$ are invertible sheaves on $X$, there is an integer $m_0$ so that  $\mathcal{E} \otimes \Lsh_m \otimes \mathcal{F}^{\sigma^m}$ is ample for all $m \geq m_0$.
\end{lemma}
\begin{proof}
(1)  Using the $\sigma$-ampleness of $\Lsh$, take $i, j \gg 0$ so that $\mathcal{M} \otimes \Lsh_i$ and $\Lsh_j \otimes \mathcal{N}^{\sigma^j}$ are globally generated.  Then $\Lsh_j^{\sigma^i} \otimes \sh{N}^{\sigma^{i+j}}$ is also globally generated.  Since the tensor product of globally generated sheaves is globally generated,  $\mathcal{M} \otimes \Lsh_{i+j} \otimes \mathcal{N}^{\sigma^{i+j}}$ is globally generated.

(2) In fact, we will show that $\mathcal{E} \otimes \Lsh_m \otimes \mathcal{F}^{\sigma^m}$ is very ample for $m \gg 0$.   Let $\mathcal{C}$ be an arbitrary very ample invertible sheaf.  By (1) we may choose $m_0$ so that if $m \geq m_0$, the sheaf
$\mathcal{K} = \mathcal{C}^{-1} \otimes \mathcal{E} \otimes \Lsh_m \otimes \mathcal{F}^{\sigma^{m}}$ is globally generated.  Since by \cite[Exercise~II.7.5(d)]{Ha} the tensor product of a very ample invertible sheaf and a globally generated invertible sheaf is very ample, $\mathcal{E} \otimes \Lsh_m \otimes \mathcal{F}^{\sigma^m} \cong \mathcal{C} \otimes \mathcal{K}$ is very ample.  
\end{proof}

\begin{proof}[Proof of Proposition~\ref{prop-left-ample}]
 Let $\mathcal{M}$ be an arbitrary coherent sheaf.  By Lemma~\ref{lem-biample}, we know that $\Ish\otimes \Lsh_n \otimes \mathcal{M}^{\sigma^n}$ is globally generated for $n \gg 0$.  We must establish that $H^j (X, \mathcal{I} \otimes \Lsh_n \otimes \mathcal{M}^{\sigma^n}) = 0$ for all $j \geq 1$ and $n \gg 0$.  

We know that   $\shTor^X_j(\struct_{\sigma^n Z}, \mathcal{M}) = 0$ for all $n \gg 0$ and   $j \geq 1$.  Thus
\[\shTor_j^X(\Ish, \mathcal{M}^{\sigma^n}) \cong \shTor^X_j(\Ish^{\sigma^{-n}}, \mathcal{M})^{\sigma^n}  \cong \shTor^X_{j+1}(\struct_{\sigma^n Z}, \sh{M})^{\sigma^n}= 0 \] 
for all $n \gg 0 $ and $j \geq 1$.

First suppose that $\mathcal{M}$ is invertible.  By Fujita's vanishing theorem \cite[Theorem~11]{Fujita1983} 
 choose an invertible sheaf $\mathcal{H}$ such that $H^i(X, \mathcal{I} \otimes \mathcal{H} \otimes \mathcal{F}) = 0$ for all $i \geq 1$ and any ample invertible sheaf $\mathcal{F}$.  By Lemma~\ref{lem-biample}(2), we may choose $m_0$ such that $\mathcal{H}^{-1} \otimes \Lsh_m \otimes \mathcal{M}^{\sigma^m}$ is ample for all $m \geq m_0$.  Then $\mathcal{I} \otimes \Lsh_m \otimes \mathcal{M}^{\sigma^m} = \mathcal{I} \otimes \mathcal{H} \otimes \mathcal{H}^{-1} \otimes \Lsh_m \otimes \mathcal{M}^{\sigma^m}$ and so its higher cohomology vanishes.

Now for general $\mathcal{M}$ let the cochain complex 
\[ \cdots \to \mathcal{P}^{-2} \to \mathcal{P}^{-1} \to \mathcal{P}^0 \to \mathcal{M} \to 0 \]
be a (not necessarily finite!) locally free resolution of $\mathcal{M}$, where each $\sh{P}^i$ is a direct sum of invertible sheaves.  By tensoring on the left with $(\mathcal{I} \otimes \Lsh_n)_{\sigma^n}$, we obtain a complex $\mathcal{Q}^{\bullet}$, where 
$\mathcal{Q}^i = \mathcal{I} \otimes \Lsh_n \otimes (\mathcal{P}^i)^{\sigma^n}$.  The $q$-th cohomology of $\mathcal{Q}^\bullet$ is isomorphic to $\shTor_{-q}^X(\Ish, \mathcal{M}^{\sigma^n}) \otimes \Lsh_n$.  Now, by  \cite[5.7.9]{Weibel}, using a Cartan-Eilenberg resolution of $\mathcal{Q}^\bullet$ we obtain  two spectral sequences 
\beq\label{ss1}
{}^{I}E^{pq}_1 =  H^q(X, \mathcal{Q}^p)
\eeq
and
\beq\label{ss2}
{}^{II} E^{pq}_2 = H^p(X, \shTor_{-q}(\mathcal{I}, \mathcal{M}^{\sigma^n})\otimes \Lsh_n).
\eeq
  Since $X$ has finite cohomological dimension $d = \dim X$, these both converge to the hypercohomology groups $\mathbb{H}^{p+q}(\mathcal{Q}^\bullet)$.  

Now, given $ p+q = j \geq 1$, by critical transversality we may take $n \gg 0 $ so that $\shTor_{-q}(\mathcal{I}, \mathcal{M}^{\sigma^n}) = 0$ for all $ q \leq -1$; thus \eqref{ss2} collapses and we obtain 
\[ \mathbb{H}^j(\mathcal{Q}^{\bullet}) = H^j(X, \mathcal{I} \otimes \mathcal{M}^{\sigma^n} \otimes \Lsh_n ).\]
On the other hand, since the sheaves $\sh{P}^i$ are locally free, applying the invertible case to each summand of $\sh{P}^i$ we may further increase $n$ if necessary to obtain that
\[H^{q} (X, \mathcal{Q}^p) = H^q(X, \sh{I} \otimes \Lsh_n \otimes (\sh{P}^p)^{\sigma^n}) =  0 \] for $d \geq q \geq 1$ and $1 -d \leq p \leq 0$.  Thus if $j \geq 1$, \eqref{ss1} collapses to 0.  Thus 
\[ H^{j} (X, \mathcal{I} \otimes \Lsh_n \otimes \mathcal{M}^{\sigma^n}) = \mathbb{H}^j(\sh{Q}^{\bullet}) = 0\]
for all $n \gg 0$ and $j \geq 1$.
\end{proof}

Lemma~\ref{lem-rtample} and Proposition~\ref{prop-left-ample}, together with Theorem~\ref{thm-VdBSerre}, will allow us to relate $\sh{R}(X, \Lsh, \sigma, Z)$ and its section ring.  To conclude this section, we show that, given $\sigma$-ampleness of $\Lsh$, the ring $R(X, \Lsh, \sigma, Z)$ is equal to the section ring of the bimodule algebra $\sh{R}(X, \Lsh, \sigma, Z)$.


\begin{lemma}\label{lem-transport}
Assume Notation~\ref{notnot}, and let $R = R(X, \Lsh, \sigma, Z)$ as in Construction~\ref{const-idealizer}.  If $\Lsh$ is $\sigma$-ample, then 
\[ R = R(X, \Lsh, \sigma, Z) = H^0(X, \sh{R}(X, \Lsh, \sigma, Z)).\]
\end{lemma}
\begin{proof}
Let $I = \Gamma_* (\sh{I})$ be the right ideal of $B(X, \Lsh, \sigma)$ generated by sections vanishing along $Z$; thus $R = \I_B(I)$.  
Suppose that $x \in R_n$, so  $xI \subseteq I$.  Since $\Lsh$ is $\sigma$-ample,  $\sh{I} \Lsh_m$ is globally generated by $I_m = H^0(X, \sh{I} \Lsh_m)$ for $m \gg 0$, and so for $m \gg 0$ 
\[ x \struct_X (\sh{I}  \Lsh^{\sigma^{m}})^{\sigma^n} = x \struct_X (I_m \struct_X)^{\sigma^n} \subseteq I_{m+n} \struct_X = \sh{I} \Lsh_{m+n}\]
for any $n$.  
Thus $x \struct_X \subseteq (\sh{I} : \sh{I}^{\sigma^n}) \Lsh_n$ and $x \in H^0(X,  (\sh{I} : \sh{I}^{\sigma^n}) \Lsh_n) = H^0(X, \sh{R}_n)$.

For the other containment, suppose that $x \in H^0(X, \sh{R}_n)$.  Then for any $m \geq 0$ we have 
\[ x I_m \subseteq H^0( (\sh{I} : \sh{I}^{\sigma^n}) \Lsh_n) \cdot H^0(X, (\sh{I} \Lsh_m)^{\sigma^n}) \subseteq H^0(X, \sh{I} \Lsh_{m+n}) = I_{n+m}.\]
Thus $x \in R_n$, and we have established the equality we seek.  
\end{proof}

\section{Noetherian idealizer rings}\label{RING}

We are now ready to begin translating our results on bimodule algebras to results on geometric idealizer rings.   We will work in the following setting:  

\begin{assnot}\label{assnotnot2}
Let $X$ be a projective variety, let $\sigma \in \Aut_{\kk} X$, and let $\Lsh$ be an invertible sheaf on $X$, which we now assume to be $\sigma$-ample.  Let $Z$ be a closed subscheme of $X$ and let $\sh{I} = \sh{I}_Z$ be its ideal sheaf.  
We continue to assume that $(\sh{I}:\sh{I}^{\sigma^n}) = \sh{I}$ for $n \gg 0$. 
Let 
\[ \sh{B} = \sh{B}(X, \Lsh, \sigma)\]
and let 
\[ B = B(X, \Lsh, \sigma).\]
Let 
\[ \sh{R} = \sh{R}(X, \Lsh, \sigma, Z) = \bigoplus_{n \geq 0} \bigl( ( \sh{I}: \sh{I}^{\sigma^n}) \Lsh_n \bigr)_{\sigma^n}.\]
  Let 
\[ I = \bigoplus_{n \geq 0} H^0( X, \sh{I} \Lsh_n) = \Gamma_* (\sh{I}) \]
and let 
\[ R = R(X, \Lsh, \sigma, Z)  = \I_B(I)\]
as in Construction~\ref{const-idealizer}.   
By Lemma~\ref{lem-transport},
\[ R = \bigoplus_{n \geq 0} H^0(X, \sh{R}_n).\]

Our assumptions imply that $\sh{R}_n = \sh{IL}_n$ and $R_n = I_n$ for $n \gg 0$.
\end{assnot}

Assume Assumption-Notation~\ref{assnotnot2}.   In this section, we determine when $R$ is left and right noetherian.  We also consider when $R$ is strongly noetherian and when $R \otimes_{\kk}R$ is noetherian.  
We first show that the right noetherian property for $R$, and in fact the strong right noetherian property, are equivalent to the simple geometric criterion from Theorem~\ref{thm-rtnoeth}.

\begin{proposition} \label{prop-SRN}  
Assume Assumption-Notation~\ref{assnotnot2}. 
 Then the following are equivalent:

$(1)$ $Z$ has finite intersection with forward $\sigma$-orbits; 

$(2)$  $R$ is right noetherian;

$(3)$ $R$ is strongly right noetherian.
\end{proposition}

\begin{proof}
(1) $\Rightarrow$ (3).  By Theorem~\ref{thm-rtnoeth}, if  (1) holds then the bimodule  algebra 
\[\sh{R}(X, \Lsh, \sigma, Z)\]
 is right noetherian.  Now let $C$ be any commutative noetherian ring, and let 
\[ X_C = X \times \Spec C\]
and
\[ Z_C = Z \times \Spec C \subseteq X_C.\]
Also define
\[B_C = B \otimes_{\kk} C,\]
\[R_C = R \otimes_{\kk} C,\]
and
\[ I_C = I \otimes_{\kk} C \cong I \otimes_{B} B_C.\]
It is clear that 
\[ R_C  = \I_{B_C}(I_C)\]
and that $R_C/I_C$ is a finitely generated $C$-module.  
Let $p: X_C \to X$ be projection onto the first factor.    

The idea behind our proof is very simple:  if $Z$ has finite intersection with forward $\sigma$-orbits, then $Z_C$ has finite intersection with forward $(\sigma \times 1)$-orbits, and so $R_C$ should be noetherian by Theorem~\ref{thm-rtnoeth} and Theorem~\ref{thm-VdBSerre}.  However, neither of these were proved  over an arbitrary base ring $C$; to work scheme-theoretically we instead follow the proof of \cite[Proposition~4.13]{ASZ1999}.  

By \cite[Proposition~4.13]{ASZ1999}, $B_C$ is noetherian.  The proof of this proposition uses the fact that the shift functor in $\rqgr B_C$ satisfies the hypotheses of \cite[Theorem~4.5]{AZ1994}. By \cite[Theorem~4.5]{AZ1994}, $B_C$ satisfies right $\chi_1$.  In particular, for any  graded right ideal $J$ of $B_C$,  the natural map
\beq\label{ROMMSEY}
 \underline{\Hom}_{B_C} (B_C/I_C, B_C/J) \to \underline{\Hom}_{\rqgr B_C} (\pi(B_C/I_C), \pi(B_C/J))
 \eeq
is an isomorphism in large degree, by \cite[Proposition~3.5]{AZ1994}.

As $\rqgr B \simeq \struct_X \lmod$, it is clear that 
\beq\label{BC}
\rqgr B_C \simeq \struct_{X_C} \lmod.
\eeq
We note that $B_C/I_C$ corresponds to $\struct_{Z_C}$ under this equivalence.

Let $J$ be a  graded right ideal  of $B_C$ containing $I_C$.     We claim that 
\[ \underline{\Hom}_{B_C}(B_C/I_C, B_C/J)\]
is a finitely generated $C$-module.  To see this, let $Y\subseteq Z_C$ be the closed subscheme of $ X_C$ such that $B_C/J$ corresponds to $\struct_Y$ under the equivalence \eqref{BC}.  By \eqref{CHARLIE}, 
$(B_C/J)[n]$ corresponds to 
\[ (\struct_Y \otimes p^* \Lsh_n)^{(\sigma^{-n} \times 1)} \cong \struct_{(\sigma^{n} \times 1)Y} \otimes p^*(\Lsh_n^{\sigma^{-n}})\]
under \eqref{BC}.  
Thus
\[ \underline{\Hom}_{\rqgr B_C} (\pi(B_C/I_C), \pi(B_C/J))_{\geq 0} \cong \bigoplus_{n \geq 0} \Hom_{X_C}(\struct_{Z_C}, \struct_{(\sigma^{n} \times 1)Y} \otimes p^*(\Lsh_n^{\sigma^{-n}})).\]
Now, $Z_C$ has finite intersection with forward $(\sigma\times 1)$-orbits, and so for $n \gg 0$, no component of $(\sigma^n \times 1) Y$ is contained in $Z_C$.  Thus
\[ \Hom_{X_C}(\struct_{Z_C}, \struct_{(\sigma^{n} \times 1)Y} \otimes p^*(\Lsh_n^{\sigma^{-n}})) = 0\]
for $n \gg 0$.  As the map \eqref{ROMMSEY} is an isomorphism in large degree, we see that
\[ \underline{\Hom}_{B_C}(B_C/I_C, B_C/ J)_n = 0\]
for $n \gg 0$, and so
\[ \underline{\Hom}_{B_C}(B_C/I_C, B_C/ J)\]
is a finitely generated $C$-module, as claimed.  As this is true for any graded $J \supseteq I_C$, by Lemma~\ref{lem-Rog-I},  $R_C$ is right noetherian.

(3) $\Rightarrow$ (2) is obvious.

(2) $\Rightarrow$ (1).   Let $x \in X$ and let $J$ be the right ideal $\Gamma_* (\sh{I}_x)$ of $B$.  As $B$ and $R$ are right noetherian, by Lemma~\ref{lem-Rog-I},
\[
\Hom_B(B/I, B/J) \cong \{ r \in B \st r I \subseteq J \} /J
\]
is a noetherian right $R/I$-module.  It is  thus finite-dimensional, as $R/I$ is finite-dimensional by assumption.  

As $\Lsh$ is $\sigma$-ample, $\Lsh_n$ is globally generated for $n \gg 0$; in particular, $J_n \subsetneqq B_n$ for $n \gg 0$.  Now, suppose that
\[ \{ n \geq 0 \st \sigma^n(x) \in Z \} = \{ n \geq 0 \st x \in \sigma^{-n}(Z) \}\]
is infinite.  For any such $n$, we have that $B_nI \subseteq J$.  Thus
\[ \{ r \in B \st rI \subseteq J \}/J \]
is infinite-dimensional, giving a contradiction.  

Thus 
$ \{ n \geq 0 \st \sigma^n(x) \in Z \}$
is finite.   
\end{proof}

The left-hand side is very different.  If $\sh{R}$ is left noetherian, then so is $R$; but  $R$ can only be strongly left noetherian if  $Z$ is a divisor.  In this case, $R$ is both a left and a right idealizer, so the strong left noetherian property will follow from the left-handed version of Proposition~\ref{prop-SRN}.

\begin{proposition}\label{prop-leftnoeth}
Assume Assumption-Notation~\ref{assnotnot2}.  
If $\{\sigma^n Z\}_{n \geq 0}$ is critically transverse, then $R= R(X, \Lsh, \sigma, Z)$ is left noetherian.
\end{proposition}
\begin{proof}
By Proposition~\ref{prop-left-ample} and Corollary~\ref{cor-CT-noeth}, we have that $\sh{R} = \sh{R}(X, \Lsh, \sigma, Z)$ is left noetherian and that $\{ (\sh{R}_n)_{\sigma^n}\}$ is a left ample sequence.  Thus by Theorem~\ref{thm-VdBSerre}, the section ring $R(X, \Lsh, \sigma, Z)$ is also left noetherian.
\end{proof}

Unfortunately, we cannot prove the converse to Proposition~\ref{prop-leftnoeth} in full generality.  We do give below several special cases where the converse does hold.

\begin{proposition}\label{prop-not-leftnoeth}
Assume Assumption-Notation~\ref{assnotnot2}.  
If $\{\sigma^n(Z)\}_{n \geq 0}$ is not critically transverse, and either

$(1)$ there is some $\sigma$-invariant subscheme $Y$ that is not homologically transverse to $Z$; or

$(2)$  all irreducible components of $Z$ have codimension 1

\noindent then $R = R(X, \Lsh, \sigma, Z)$ is not left noetherian.
\end{proposition}

Before giving the proof, we give a preliminary lemma.

\begin{lemma}\label{lem-divisor}
Let $X = X^{(0)} \supset X^{(1)} \supset X^{(2)} \supset \cdots$ be the singular stratification of $X$.  
Suppose that $Z$ is a subscheme of pure codimension 1 such that for all $j$, $\shTor^X_1(\struct_Z, \struct_{X^{(j)}}) = 0$.  Then $Z$ is locally principal.
\end{lemma}

\begin{proof}
   Fix $x  \in Z$; we will show that $Z$ is locally principal at $x$.    Let $\struct = \struct_{X,x}$.  
   
   Let $j$ be maximal so that $x \in X^{(j)}$, and let $J$ be the ideal of $X^{(j)}$ in $\struct$.  Let $I$ be the defining ideal of $Z$ in $\struct$.   By Lemma~\ref{lem-contain-Tor}, $I \not\subseteq J$.  Thus $(I+J)/J$ locally defines a hypersurface in $X^{(j)}$.  Since $\struct/J$ is a regular local ring, $(I+J)/J$ is principal in $\struct/J$, and so there is $f \in I$ such that $(f) + J = I+J$. 
   
    By homological transversality, $I \cap J  = IJ$.  Thus
    \[ \frac{I}{(f)} \otimes_{\struct} \frac{\struct}{J} \cong \frac{I}{(f) + IJ} = \frac{I}{(f) + I\cap J}.\]
    But 
        \[ (f) + I \cap J = I \cap ((f) + J) = I \cap (I + J) = I.\]
    Thus 
    \[ \frac{I}{(f)} \otimes_{\struct} \frac{\struct}{J} \cong \frac{I}{I} = 0.\]

Let $K$ be the residue field of $\struct$.  Since $I/(f)\otimes_{\struct} ( \struct/J)$ surjects on to 
  $(I/(f)) \otimes_\struct K$ we see that $(I/(f)) \otimes_\struct K=0$.  Nakayama's Lemma implies that $I = (f)$.
\end{proof}

\begin{proof}[Proof of Proposition~\ref{prop-not-leftnoeth}]
Suppose (1) holds.  Let  $Y$ be a $\sigma$-invariant subscheme that is not homologically transverse to $Z$, and let $j \geq 1$ be such that 
\[\shTor^X_j(\struct_Z, \struct_Y) \neq 0.\]
  Let $\sh{J} = \sh{I}_Y$, and let $J = \Gamma_*(\sh{J})$ be the right ideal of $B$ generated by sections that vanish on $Y$.   Since $\sigma Y=Y$, $J$ is a two-sided ideal of $B$.
We claim that $\Tor^B_j(B/I, B/J)_n \neq 0$ for $n \gg 0$.  

Form a graded projective resolution
\[ \cdots \to P^{-1} \to P^0 \to B/I \to 0 \]
of the right $B$-module  $B/I$, where each $P^i$ is a finitely generated graded free right $B$-module.  Thus  for each $i \leq 0$, there is a finite multiset $A_i$ of integers such that
\[ P^i = \bigoplus_{a \in A_i} B[a].\]
Now, for each $i$ let $\sh{P}^i = \widetilde{ P^i} $. Since the functor $\,\, \widetilde{}\,\,$ is exact, the  complex
\[ \cdots \to  \sh{P}^{-1} \to \sh{P}^0 \]
is  a resolution of $\struct_Z = \widetilde{B/I}$.
Furthermore, by the $\sigma$-invariance of $Y$ and the $\sigma$-ampleness of $\Lsh$, for $-j-d \leq i \leq -j+1$ and for $n \gg 0$, we have that
\beq\label{eq-save}
 H^0(X, \sh{P}^i \otimes \Lsh_n \otimes \struct_Y) = \bigoplus_{a \in A_i} (B/J)_{n+a} = (P^i \otimes_B B/J)_n.
 \eeq

Fix $n$ and let $\sh{Q^\bullet} = \sh{P}^\bullet \otimes \Lsh_n \otimes \struct_Y$.  
Reasoning as in Proposition~\ref{prop-left-ample}, from a Cartan-Eilenberg resolution $\sh{C}^{\bullet, \bullet}$ of $\sh{Q}^{\bullet}$ we obtain two spectral sequences
\beq \label{ss-shadow}
{}^{I} E^2_{pq} = H^p({H}^q(X, \sh{Q}^{\bullet})) 
\eeq
and
\beq \label{ss-kiri}
{}^{II} E^2_{pq}  = {H}^p(X, \shTor^X_{-q}(\struct_Z , \struct_Y) \otimes \Lsh_n),
\eeq
both of which converge (since $X$ has finite cohomological dimension) to the hypercohomology $\mathbb{H}^{p+q}(\sh{C}^{\bullet, \bullet})$.

By $\sigma$-ampleness of $\Lsh$, by taking $n \gg 0$  we may assume that
\[ {H}^p(X, \shTor_{-q}(\struct_Z, \struct_Y) \otimes \Lsh_n) = 0 \quad \mbox{for $p \geq 1$ and $-j-d \leq q \leq -j -1$} \]
and that
\[ {H}^q(X, \sh{Q}^p) = 0 \quad \mbox{ for $q \geq 1$ and $-j-d \leq p \leq -j-1$}.\]
Thus for $p+q = -j$, both \eqref{ss-shadow} and \eqref{ss-kiri} collapse, and we obtain that
\beq \label{eq-tabitha}
{H}^0(X, \shTor_{j}^X(\struct_Z, \struct_Y) \otimes \Lsh_n) = H^{-j} ({H}^0(X, \sh{Q}^{\bullet})).
\eeq
Since $\shTor_j^X(\struct_Z, \struct_Y) \neq 0$ and $\Lsh$ is $\sigma$-ample, for $n \gg 0$ the left-hand side of \eqref{eq-tabitha} is nonzero; but \eqref{eq-save} implies that for $n \gg 0$, the right-hand side is equal to
 \[ H^{-j}(P^\bullet \otimes_B B/J)_n = \Tor^B_j(B/I, B/J)_n.\]
 Thus $\Tor^B_j(B/I, B/J)_n \neq 0$.
 
 But if $R$ is left noetherian, then, using Proposition~\ref{prop-R-Tor} and a similar argument to that used in the proof of Lemma~\ref{lem-Tor-reduction}, for any finitely generated left $B$-module $M$ and for any $j \geq 1$, we must have that $\Tor^B_j(B/I, M)$ is torsion.   Since we have shown this is false for $M = B/J$, $R$ is not left noetherian.

Now suppose that (2) holds.  Consider the singular stratification 
\[ X = X^{(0)} \supset X^{(1)} \supset \cdots \]
of $X$.  If $Z$ is not homologically transverse to some $X^{(i)}$, then by (1)  $R$ is not left noetherian.  If  $Z$ is homologically transverse to all  $X^{(i)}$, then by Lemma~\ref{lem-divisor}, $Z$ is locally principal.  By Lemma~\ref{lem-Tor-reduction}, there is some  reduced and irreducible subscheme $Y$ such that $\shTor^X_1(\struct_{\sigma^n Z}, \struct_Y) \neq 0$  for infinitely many  $n \geq 0$.   Our assumptions on $Y$ and $Z$ imply that $\shTor^X_1(\struct_{\sigma^n Z}, \struct_Y) \neq 0$ if and only if $\sigma^n Z \supseteq Y$.

Thus $\sigma^n(Z) \supseteq Y$ for infinitely many $n \geq 0$.  Let $\sh{J}$ be the ideal sheaf defining $Y$ and let
  \[A = \{ n \geq 0 \st Y \subseteq \sigma^n Z\}= \{ n \geq 0 \st \sh{J}^{\sigma^n} \supseteq \sh{I} \}. \]
  Let $R' = \kk \oplus H^0(X, \sh{IB}_+)$.  It is sufficient to show that $R'$ is not left noetherian.  
    Let 
\[ J = \bigoplus_{n \geq 0} H^0(X, (\sh{I} \cap \sh{J}^{\sigma^n}) \Lsh_n).\]
We will show that the left ideal $J$ of $R'$  is not finitely generated. 

 Fix an integer $k \geq 1$.  By $\sigma$-ampleness of $\Lsh$, we may choose $n > k$ so that $n \in A$ and $(\sh{I} \cap \sh{J}^{\sigma^n}) \Lsh_n = \sh{I}\Lsh_n$ is globally generated.    Then
\[ (R' \cdot J_{\leq k})_n \subseteq H^0(X, \sh{IJ}^{\sigma^n} \Lsh_n) \subsetneqq J_n\]
and ${}_{R'} J$ is not finitely generated.
\end{proof}

As a corollary to Proposition~\ref{prop-not-leftnoeth}, we obtain that if we are in characteristic  0 and if $\sigma$ is an element of an algebraic group acting on $X$, then the converse to Proposition~\ref{prop-leftnoeth} holds.

\begin{corollary}\label{cor-ringsnoeth}
Suppose that  $\chrr \kk =0$.  Let $X$ be a projective variety and let $\sigma$ be an element of an algebraic group $G$ that acts on $X$.  Let $\Lsh$ be a $\sigma$-ample invertible sheaf on $X$.  Let $Z$ be a  closed subscheme of $X$ such that the components of $Z^{\red}$ have infinite order under $\sigma$.  Then the following are equivalent:

$(1)$ the geometric idealizer  $R = R(X, \Lsh, \sigma, Z)$ is noetherian;

$(2)$ the set $\{ \sigma^n Z\}_{n \geq 0}$ is critically transverse;

$(3)$  $Z$ is homologically transverse to all reduced $\sigma$-invariant subschemes of $X$.
\end{corollary}

\begin{proof}
First suppose that there is $x \in X$ so that $\{ n \geq 0 \st \sigma^n(x) \in Z\}$ is infinite.  Then by Proposition~\ref{prop-SRN}, $R$ is not right noetherian.  Furthermore, $\{\sigma^n Z\}_{n \in \ZZ}$ is certainly not critically transverse, and so by Theorem~\ref{thm-whenCT} there is a  $\sigma$-invariant subscheme that is not homologically transverse to $Z$, and $\{\sigma^nZ\}_{n \geq 0}$ is not critically transverse.  Thus $(1)$, $(2)$, and $(3)$ all fail, and the result holds.

 Thus we may assume that no such $x$ exists; by Proposition~\ref{prop-SRN}, $R$ is right noetherian.  Note also that by Lemma~\ref{lem-niceconds}, Assumption-Notation~\ref{assnot-FOO} is satisfied.  
Then $(1) \Rightarrow (3)$ is  Proposition~\ref{prop-not-leftnoeth}. $(3) \Rightarrow (2)$ is  Theorem~\ref{thm-whenCT}.  $(2) \Rightarrow (1)$ is  Proposition~\ref{prop-leftnoeth}.  
\end{proof}

Since the geometric condition required for a right idealizer to be left noetherian is fairly subtle, it is not surprising that right idealizers are almost never strongly left noetherian.   To show this, we use the concept of {\em generic flatness}\index{generically flat}, as defined in \cite{ASZ1999}.  Let $C$ be a commutative noetherian domain.  We say that a $C$-module $M$ is {\em generically flat} if there is some $f \neq 0 \in C$ such that $M_f$ is flat over $C_f$.  If $R$ is a finitely generated commutative $C$-algebra, then by Grothendieck's generic freeness theorem \cite[Theorem~6.9.1]{EGA-IV.2}, every finitely generated $R$-module is a generically flat $C$-module.

Artin, Small, and Zhang have generalized this result to strongly noetherian noncommutative rings.  They prove:
\begin{theorem}\label{thm-ASZ}
{\em (\cite[Theorem~0.1]{ASZ1999})}
Let $R$ be a strongly noetherian algebra over an excellent Dedekind domain $C$.  Then every finitely generated right $R$-module is generically flat over $C$. \qed
\end{theorem}

\begin{lemma}\label{lem-noGF}
Assume Assumption-Notation~\ref{assnotnot2}.  
If $Z'$ is a component of $Z$ such that $\codim Z' \geq 2$ and such that $\bigcup_{m \geq 0} \sigma^m Z'$ is Zariski dense in $X$, then for every open affine $U \subseteq X$,  the finitely generated left $R \otimes_k \struct(U)$-module 
\[M = \bigoplus_{n \geq 0} \mathcal{R}(\sigma^{-n} U)\]
is not a generically flat $\struct(U)$-module.
\end{lemma}

\begin{proof}
We first verify that $M$ is a left $R$-module.   By \cite[Equation~2.5]{AV}, the multiplication rule in $\mathcal{R}$ acts on sections via:
\[ \mathcal{R}_n(V)  \times \mathcal{R}_m(\sigma^n V) \to \mathcal{R}_{n+m}(V)\]
or, writing $V = \sigma^{-n-m}U$, 
\[ \mathcal{R}_n (\sigma^{-n-m}U) \times \mathcal{R}_m(\sigma^{-m}U) \to \mathcal{R}_{n+m}(\sigma^{-n-m}U).\]
Thus we have a map
\[ R_n \times M_m = \mathcal{R}(X) \times \mathcal{R}_m(\sigma^{-m}U) \stackrel{{\rm res}}{\to} \mathcal{R}_n (\sigma^{-n-m}U) \times \mathcal{R}_m(\sigma^{-m}U) \to M_{m+n}.\]
Verifying associativity is trivial, and so $M$ is a left $R$-module.

Let $C = \struct(U)$.  By identifying $C$ with $C^{\rm op}$, consider the right action of $C$ on $M$ given by $g \star f = g \cdot f^{\sigma^n} = g \cdot ( f \circ \sigma^n)$, where $g \in M_n$, $f \in C$.   Note that  $f \circ \sigma^n$ acts on $\sigma^{-n} U$ and so does act naturally on  elements of $M_n$.

Now since for $n \gg 0$ the sheaves $\mathcal{I} \otimes \Lsh_n$ are globally generated, the restriction map $R \to M$ is surjective in degree $\geq m$ for some $m$.  But since $M_{<m}$ is a finitely generated $C$-module, therefore $M$ is  a finitely generated $R_C$ module.  

Now let $f$ be an arbitrary element of $C$; let $M' = M_f$.  Since $\bigcup_{m \geq 0} \sigma^m(Z')$ is Zariski dense, there is some $m$ such that $\sigma^m Z'$ meets $U_f$, say at a point $p$.  But then $(M'_m)_p = (\Ish {\Lsh_m})_{\sigma^{-m}p}$, which is not flat over $C_{p}$, since codim $Z' \geq 2$.  Thus $M_f$ is not flat over $C_f$.  
\end{proof}

\begin{corollary}\label{cor-SLN}
Assume Assumption-Notation~\ref{assnotnot2}.  Then 
$R$ is strongly left noetherian if and only if $Z$ has pure codimension 1 and $\{ \sigma^n Z \}_{n \geq 0}$ is critically transverse in $X$.
\end{corollary}
\begin{proof}
If  $\{\sigma^n Z\}_{n \geq 0}$ is critically transverse, then in particular $Z$ is homologically transverse to the singular stratification of $X$.  If $Z$ has pure codimension 1, then  by Lemma~\ref{lem-divisor}, $Z$ is locally principal and $\sh{I} = \sh{I}_Z$ is invertible.  Now, letting $\Lsh' = \sh{I} \Lsh (\sh{I}^{-1})^\sigma$, we have that $\sh{I} \Lsh_n = (\Lsh')_n \sh{I}^{\sigma^n}$.  Since $\Lsh'$ is clearly also $\sigma$-ample, we see that $R$ is also the {\em left} idealizer at $Z$ inside the twisted  homogeneous coordinate ring $B(X, \Lsh', \sigma)$.  By assumption on critical transversality, we have in particular that for any $p \in X$, the set $\{ n \leq 0 \st \sigma^n(p) \in Z \}$ is finite.  Thus by  Proposition~\ref{prop-SRN}, $R$ is strongly left noetherian.

If $ Z$ has pure codimension 1 and $\{ \sigma^n Z\}_{n \geq 0}$ is not critically transverse, then by Proposition~\ref{prop-not-leftnoeth}(2), $R$ is not left noetherian so is certainly not strongly left noetherian.

Now suppose that $Z$ has a component $Z'$ such that $\codim Z' >1$.  Let $Y$ be the Zariski closure of  $\{ \sigma^n (Z')\}_{n \geq 0}$.  Note that there is a chain of containments
\[ Y \supseteq \sigma Y \supseteq \sigma^2 Y \supseteq \cdots.\]
As this chain terminates, for some $j$ we have $\sigma^j Y = \sigma^{j+1} Y$; applying $\sigma^{-j}$ we see that $Y$ itself is $\sigma$-invariant.   If $Y \neq X$, then $Z$ is not homologically transverse to $Y$ by Lemma~\ref{lem-contain-Tor}, and so by Proposition~\ref{prop-not-leftnoeth}(1), $R$ is not left noetherian. 

We have thus reduced to considering the case that $Z$ has a component $Z'$ such that $\codim Z' > 1$ and that $\bigcup_{n \geq 0} \sigma^n(Z')$ is Zariski-dense in $X$.   
Fix an open affine $U \subseteq X$ such that $X \smallsetminus U$ has codimension 1.  Let $M$ be the module from Lemma~\ref{lem-noGF}.  As $M$ is not a generically flat left $\struct(U)$-module, by Theorem~\ref{thm-ASZ}, $R \otimes_{\kk} \struct(U)$ is not strongly left noetherian.  Thus $R$ is not strongly left noetherian.  
  \end{proof}

To end this section, we use Proposition~\ref{prop-not-leftnoeth} to show that the idealizers $R = R(X, \Lsh, \sigma, Z)$ have the unusual property that $R \otimes_{\kk} R$ is not noetherian.  The idealizers constructed by Rogalski in \cite{R-idealizer} were the first known examples of noetherian rings with this property.

\begin{proposition}\label{prop-tensor}
Assume Assumption-Notation~\ref{assnotnot2}.  Assume in addition that $Z$ is not of pure codimension 1.  Then $R \otimes_{\kk} R$ is not left noetherian.  
\end{proposition}
\begin{proof}
Let $T = R\otimes_{\kk} R$.  Consider the $\ZZ$-grading on $T$ given by 
\[ T_n = \sum_{i \in \NN} R_i \otimes R_{n+i}.\]
If $T$ is left noetherian, then $T_0$ must also be left noetherian.  Thus we will show that $T_0$ is not left noetherian.

Let $S = T_0$. The ring $S$ is also known as the {\em Segre product} of $R$ with itself, and is sometimes written  $S = R \segre R$.  It has a natural $\NN$-grading, given by $S_n = R_n \otimes R_n$.  Let $Y = (Z \times X) \cup (X \times Z)$.  It can readily be seen that 
\[ S = R(X \times X, \Lsh \boxtimes \Lsh, \sigma \times \sigma, Y).\]

Let $\Delta \subseteq X\times X$ be the diagonal.  We claim that $Y$ is not homologically transverse to $\Delta$.  As $R$ is left noetherian, by Proposition~\ref{prop-not-leftnoeth}, $Z$ is homologically transverse to the singular stratification of $X$.  In particular, there is some component $Z'$ of $Z$, of codimension $d \geq 2$,  so that $Z'$ is not contained in the singular stratification of $X$.  Thus, if $\eta$ is the generic point of $Z'$, the local ring $\struct = \struct_{X, \eta}$ is a regular local ring of dimension $d$.   Let $J$ be the ideal of $Z$ in $\struct$.   By assumption, $J$ is $\eta$-primary and is not principal.    

Let $\struct' = \struct_{X \times X, \eta \times \eta}$.  There is a natural embedding $\struct \otimes_{\kk} \struct \to \struct'$.  Let $x_1, \ldots x_d$ be a system of parameters for the maximal ideal $\eta_\eta$ of $\struct$.  The maximal ideal   $\mf{m}$ of $\struct'$ is generated by the system of parameters $x_1 \otimes 1, \ldots, x_d \otimes 1, 1 \otimes x_1, \ldots, 1 \otimes x_d$.   Let $K$ be the ideal of $Y$ in $\struct'$; note that $K$ is  generated by the image of $J \otimes J$. 

Let $H = (\{x_i \otimes 1 - 1 \otimes x_i\st 1 \leq i \leq d \})$ be the ideal of $\Delta $ in $\struct'$.  Now, by \cite[ Theorem~A.1.1]{Ha}, the intersection product on $X \times X$ satisfies  
\[ i(Y, \Delta; \eta) = i(X \times Z, \Delta; {\eta} ) + i(Z \times X, \Delta; \eta).\]
This is easily seen to be 
\[ 2 \len_{\eta} (\struct/J).\]

On the other hand, consider $\struct'/K \otimes_{\struct'} \struct'/H \cong \struct'/(K + H)$.  Since $\struct'/H \cong \struct$, this is isomorphic to 
$\struct/J^2$.  As $J$ is not principal, we have
\beq\label{ANI}
 \len_{\eta} (\struct/J^2 ) > 2 \len_{\eta} (\struct/J).
\eeq
By \eqref{intprod}, 
\[ 2 \len_{\eta} (\struct/J) = i(Y, \Delta; \eta) = \sum_{i \geq 0} (-1)^i \len_{\eta} \shTor^{X \times X}_i(\struct_Y, \struct_{\Delta}).\]
Using \eqref{ANI}, we obtain
\begin{multline*}
 \sum_{i \geq 1} (-1)^i \len_{\eta} \shTor^{X \times X}_i(\struct_Y, \struct_{\Delta}) 
= i(Y, \Delta; \eta) - \len_{\eta}(\struct_Y \otimes \struct_{\Delta}) \\
= 2 \len_{\eta} \struct/J - \len_{\eta} \struct/J^2 < 0.
\end{multline*}
Thus $Y$ and $\Delta$ are not homologically transverse. 

Now, $\Delta$ is $\sigma \times \sigma$-invariant, and the proof of  \cite[Corollary~3.5]{Keeler2003} shows that $\Lsh\boxtimes \Lsh$ is $\sigma\times\sigma$-ample.  Thus, by Proposition~\ref{prop-not-leftnoeth}, 
\[ T_0 = S \cong  R(X \times X, \Lsh \boxtimes \Lsh, \sigma \times \sigma, Y)\]
is not left noetherian.  Thus $T$ is not left noetherian.
\end{proof}

We remark that if $Z$ is reduced at $\eta$ (so that $J = \eta$), then it is easy to see directly that $\shTor_1^{X \times X} (\struct_Y, \struct_{\Delta}) \neq 0$; indeed, the element  
\[ x_1  \otimes x_2 - x_2 \otimes x_1 = (x_1 \otimes 1 - 1 \otimes x_1) (x_2 \otimes 1) + (1 \otimes x_2 - x_2 \otimes 1)(x_1 \otimes 1)\]
is in $H \cap K \smallsetminus HK$; note that $HK \subseteq \mf{m}^3$.

\section{The $\chi$ conditions for idealizers}\label{CHI}

In this section, we determine the homological properties of graded idealizers; specifically, we investigate the Artin-Zhang $\chi$ conditions, as defined in the introduction.

We first recall Rogalski's result that a right idealizer will fail $\chi_1$ and all higher $\chi_j$ on the left.  
\begin{proposition}\label{prop-left-chi}
{\em (Rogalski)}
Assume Assumption-Notation~\ref{assnotnot2}.     Then $R$ fails left $\chi_1$.
\end{proposition}
\begin{proof}
This is proved in  \cite[Proposition~4.2]{R-idealizer}.  
To see it directly, note that changing $R$ by a finite-dimensional vector space does not affect the $\chi$ conditions, so without loss of generality we have $R = \kk + I$.  Now  $B/R$ is infinite-dimensional and is killed on the left by $I$; as we have an injection $B/R \hookrightarrow \Ext^1_R(\kk, R)$, we see that  $\Ext^1_R(\kk, R)$ is infinite-dimensional.  
\end{proof}

To analyze the right $\chi$ conditions, our key result is the following, due to Rogalski:

\begin{proposition}\label{prop-R-chi}
{\em (\cite[Proposition~4.1]{R-idealizer})}
Let $B$ be a noetherian ring that satisfies right $\chi$.  Let $I$ be  a a right ideal of $B$, and let $R = \I_B(I)$.  Assume that $B/I$ is infinite-dimensional, that $B_R$ is finitely generated, and that $R/I$ is finite-dimensional.  Then $R$ satisfies right $\chi_i$ for some $i \geq 0$ if and only if 
$\underline{\Ext}^j_B(B/I, M)$ is finite-dimensional for all $0 \leq j \leq i$ and all $M \in \rgr B$. \qed
\end{proposition}

Rogalski proved that the right idealizer of a point in $\PP^d$ satisfies right $\chi_{d-1}$ and fails right $\chi_d$ if the orbit of the point is critically dense.   Here we extend Rogalski's result to idealizers at higher-dimensional subschemes.

\begin{lemma}\label{lem-whatsExt}
Let $X$ be a projective variety, let $\sigma \in \Aut_{\kk} X$, and let $\Lsh$ be a $\sigma$-ample invertible sheaf on $X$.  
Let $Z$ and $Y$ be closed subschemes of $X$, and let $B = B(X, \Lsh, \sigma)$.  Let  $J$ be the right ideal of $B$ consisting of sections vanishing along $Y$, and let $I$ be the right ideal of $B$ consisting of sections vanishing along $Z$.  For $n \gg 0$, there is an isomorphism of $\kk$-vector spaces
\[ \underline{\Ext}^j_{\rgr B}(B/I, B/J)_n \cong \Ext^j_X(\struct_Z,  \struct_{\sigma^n Y} \otimes \Lsh_n^{\sigma^{-n}}).\]
\end{lemma}
\begin{proof}
The twisted homogeneous coordinate ring $B$ satisfies $\chi$ as a consequence of 
\cite[Theorem~7.3]{Y1992} and  \cite[Theorem~6.3]{VdB1997} (or alternately, \cite[Theorem 4.2]{YZ1997}).  Thus the  natural map from $\underline{\Ext}^j_{\rgr B}(B/I, B/J)$ to $\underline{\Ext}^j_{\rQgr B}(\pi(B/I), \pi(B/J))$ has right bounded kernel and cokernel by \cite[Proposition~3.5]{AZ1994}.  It therefore suffices to show that for $n \gg 0$, we have 
\[ \underline{\Ext}^j_{\rQgr B}(\pi(B/I), \pi(B/J))_n \cong \Ext^j_X(\struct_Z, \struct_{\sigma^n Y} \otimes \Lsh_n^{\sigma^{-n}}).\]
In fact, we show that we have this isomorphism for all $n$.

Using the equivalence between $\rqgr B$ and $\struct_X \lmod$, we have that
\begin{multline*}
 \underline{\Ext}^j_{\rQgr B}(\pi(B/I), \pi(B/J))_n \cong {\Ext}^j_{\rQgr B}(\pi(B/I), \pi((B/J) [n])) \\
 \cong
\Ext^j_X(\widetilde{B/I}, \widetilde{(B/J)[n]}).
\end{multline*}
Now, $\widetilde{B/I} =  \struct_Z$, and by \eqref{CHARLIE},
\[ \widetilde{(B/J)[n]} \cong (\struct_Y \Lsh_n)^{\sigma^{-n}} \cong \struct_{\sigma^n Y} \Lsh_n^{\sigma^{-n}}.\]
 The result follows.
\end{proof}

We have seen that for $R$ to be right noetherian is relatively straightforward, but the left noetherian property for $R$ depends on the critical transversality of $\{ \sigma^n Z\}$.  It turns out that the right $\chi_j$ properties, for $j \geq 1$, also depend on the critical transversality of $\{ \sigma^n Z\}$.  In particular, we have:

\begin{proposition}\label{prop-chi}
Assume Assumption-Notation~\ref{assnotnot2}.  Let $d = \codim Z$.  

$(1)$ If $\{\sigma^n Z\}_{n \leq 0}$ is critically transverse, and either

{\rm(a)} $X$ is nonsingular and $Z$ is Gorenstein; or

{\rm(b)} $Z$ is 0-dimensional,

\noindent then $R$ satisfies right $\chi_{d-1}$ but fails right $\chi_d$.

$(2)$ More generally, if $Z$ contains an irreducible component of codimension $d$ that is not contained in the singular locus of $X$, then $R$ fails right $\chi_d$.  In particular, if $R$ is left noetherian then $R$ fails right $\chi_d$.
\end{proposition}
\begin{proof}
 By Proposition~\ref{prop-R-chi}, $R$ satisfies right $\chi_i$ if and only if  for all finitely generated $M_B$
  we have $\dim_{\kk} \underline{\Ext}_{\rgr B}^j(B/I, M) < \infty$ for all $j \leq i$.  Furthermore, using the equivalence of categories between $\rqgr B$ and $\struct_X \lmod$, without loss of generality we may assume that $M = B/J$, where $J$ is a right ideal of $B$ consisting of sections vanishing along a reduced, irreducible subscheme $Y$ of $X$.  

 By Lemma \ref{lem-whatsExt}, for $n \gg 0$ we have isomorphisms 
\[\underline{\Ext}_{\rgr B}^j(B/I), B/J)_n \cong \Ext_X^j(\struct_Z, \struct_{\sigma^n Y} \otimes \Lsh_n^{\sigma^{-n}}). \]
  Thus we have:
 \begin{multline} \label{foo}
  \mbox{ $R$ satisfies right $\chi_i$ $\iff$ for all $Y \subseteq X$,} \\
  \mbox{$ \Ext^j_X(\struct_Z, \struct_{\sigma^n Y} \otimes \Lsh_n^{\sigma^{-n}}) = 0$ for all $j \leq i$ and $n \gg 0$.}
  \end{multline}
  
By \cite[Prop~4.2.1]{Tohoku}, for any coherent sheaves $\sh{E}$ and $\sh{F}$  there is a spectral sequence
\beq \label{ss-Gro}
H^p(X, \shExt^q_X(\sh{E}, \sh{F})) \Rightarrow \Ext^{p+q}_X(\sh{E}, \sh{F}).
\eeq
We consider the special case
\begin{equation}\label{ss3} 
E^{pq} = H^p (X, \shExt_X^q(\struct_Z, \struct_{\sigma^n Y} \otimes \Lsh_n^{\sigma^{-n}})) \Rightarrow \Ext_X^{p+q} (\struct_Z, \struct_{\sigma^n Y} \otimes \Lsh_n^{\sigma^{-n}}). 
\end{equation}
We first suppose  that (1)(a) holds, and show that $R$ satisfies right $\chi_{d-1}$.  

  Fix a closed subscheme $Y$ of $X$ and consider the sheaf $\shExt_X^j(\struct_Z, \struct_{\sigma^n Y})$.  This is supported on $Z$; we compute it by working locally at some closed point $x \in Z$.   Gorenstein rings are Cohen-Macaulay and therefore locally equidimensional \cite[Corollary~18.11]{Eis}, so we may assume that $Z$ is pure-dimensional of codimension $d' \geq d$.  Let $J \subseteq \struct$ be the ideal defining $Z$ locally at $x$.  

By \cite[Corollary~21.16]{Eis},   $\struct/J$ has a self-dual free resolution as an $\struct$-module
\[ 0 \to Q_{d'} \to \cdots \to Q_0 \to \struct/J. \]
We write this resolution as  $Q_{\bullet} \to \struct/J$.

 For a given $n$, let $K \subseteq \struct$ be the ideal defining $\sigma^nY$ at $P$.  Let $M = \struct/K$. Then we have isomorphisms of complexes
  \[\Hom_{\struct}(Q_\bullet, M) \cong \Hom_{\struct}(Q_{\bullet}, \struct) \otimes M \cong Q_{\bullet} \otimes M,\]
  where the final isomorphism follows from the fact that $Q_{\bullet}$ is self-dual. 
  The right-hand complex of this equation computes $\Tor^{\struct}_{d'-j}(\struct/J, M)$.  Thus
   we obtain isomorphisms
\beq \label{Ext-Tor}
\shExt_X^j(\struct_Z, \struct_{\sigma^n Y}) \cong \shTor_{d'-j}^X(\struct_{Z}, \struct_{\sigma^n Y}) \cong \shTor_{d'-j}^X(\struct_{\sigma^{-n}Z}, \struct_Y)^{\sigma^{-n}}
\eeq
for all $j$.

We return to the Grothendieck spectral sequence \eqref{ss3}.   By \cite[III.6.7]{Ha},  
\[\shExt^q_X(\struct_Z, \struct_{\sigma^n Y} \otimes \Lsh_n^{\sigma^{-n}}) \cong \shExt^q_X(\struct_Z, \struct_{\sigma^n Y}) \otimes \Lsh_n^{\sigma^{-n}}.\]
  Using critical transversality and \eqref{Ext-Tor}, choose $n_0$ such that
$\shExt_X^j(\struct_Z, \struct_{\sigma^n Y}) = 0$ for all $n \geq n_0$ and  $j < d \leq d'$. Then $E^{pq} = 0$ for $q < d$; so we see that if $p+q  = j < d$, then \eqref{ss3} collapses to 0 and we have
$\Ext^j_X(\struct_Z, \struct_{\sigma^n Y} \otimes \Lsh_n^{\sigma^{-n}}) = 0$ for $n \gg 0$.  By \eqref{foo}, $R$ satisfies $\chi_{d-1}$.

Let $X^{\sing}$ be the singular locus of $X$.  We now suppose that (2) holds; that is, $Z$ contains an irreducible component of codimension $d$ that is not contained in $X^{\sing}$.   We show that in this situation, $R$ fails right $\chi_d$.

We  consider the special case of \eqref{ss3} where $Y = X$:  
\beq\label{ss4}
H^p(X, \shExt^q_X(\struct_Z, \Lsh_n^{\sigma^{-n}})) \Rightarrow \Ext^{p+q}_X(\struct_Z, \Lsh_n^{\sigma^{-n}}).
\eeq
Let $x \in Z$ be a nonsingular point of $X$  such that  the codimension of $Z$ at $x$ is $d$.  
  Since $X$ is nonsingular at $x$, by \cite[Theorem~1.2.5]{BH1993} 
  \beq\label{apple} \shExt^j_X(\struct_Z, \struct_X)_x = 0 \quad \mbox{ for $j < d$}
  \eeq
   and 
   \beq\label{pear}
   \shExt^d_X(\struct_Z, \struct_X)_x \neq 0.
   \eeq 
   Now \eqref{apple} implies that for $p+q = d$, \eqref{ss4} collapses, and we obtain that
   \[ \Ext^d_X(\struct_Z, \Lsh_n^{\sigma^{-n}}) \cong H^0\bigl( X, \shExt^d_X(\struct_Z, \struct_X) \otimes \Lsh_n^{\sigma^{-n}}\bigr) \cong H^0\bigl(X, \shExt^d_X(\struct_Z, \struct_X)^{\sigma^{n}} \otimes \Lsh_n\bigr).\]
   This is nonzero for $n \gg 0$ by \eqref{pear} and $\sigma$-ampleness of $\Lsh$.  
   Thus by \eqref{foo}, $R$ fails right $\chi_d$.
   
   We have seen that if (2) holds, then $R$ fails right $\chi_d$.  We note that if $\{\sigma^n Z\}_{n \leq 0}$ is critically transverse, then $Z$ is homologically transverse to all $\sigma$-invariant subschemes, and certainly no component of $Z$ is contained in $X^{\sing}$.   If $R$ is left noetherian, then using Proposition~\ref{prop-not-leftnoeth} and Lemma~\ref{lem-contain-Tor}, we again  have that no component of $Z$ is contained in the singular locus of $X$.  Thus if (1)(a) or (1)(b) hold, or if $R$ is left noetherian, then (2) holds and $R$ fails right $\chi_d$.
   
   It remains to show that if (1)(b) holds, then $R$ satisfies right $\chi_{d-1}$.
    We have seen that $X$ is nonsingular at all points of $Z$, and so \eqref{apple} holds.   Let $j \leq d-1$.
   By \eqref{ss4} we have that $\Ext^j_X(\struct_Z, \Lsh_n^{\sigma^{-n}}) = 0$.  On the other hand, if $Y \subset X$ is a proper subvariety, then critical transversality of $\{ \sigma^n Z \}_{n \leq 0}$ and Corollary~\ref{cor-CT=CD} show that $\sigma^n Y$ and $Z$ are disjoint for $n \gg 0$, and so certainly $\Ext^j_X(\struct_Z, \struct_{\sigma^nY} \otimes \Lsh_n) = 0$ for $n \gg 0$.  By \eqref{foo}, $R$ satisfies right $\chi_{d-1}$.
\end{proof}

\section{Proj of graded idealizer rings and cohomological dimension}\label{COHDIM}
 Assume Assumption-Notation~\ref{assnotnot2}.  We end this paper by investigating the cohomological dimension of the (right) noncommutative projective scheme associated to $R$; we briefly review the definitions here.  
 
Let $R$ be a (noncommutative) $\NN$-graded ring, and recall that the category $\rQgr R$ is  the noncommutative analogue of Proj of a commutative graded ring.  In \cite{AZ1994}, Artin and Zhang define $\rProj R$ to be the pair $(\rQgr R, \pi R)$, where $\pi: \rGr R \to \rQgr R$ is the quotient functor.  
  The cohomology groups on $\rProj R$ are defined by setting
 \[ H^i(\rProj R, \sh{M}) = \Ext^i_{\rQgr R}(\pi R, \sh{M})\]
  for any $\sh{M} \in \rQgr R$.  The {\em cohomological dimension}\index{cohomological dimension} of $\rProj R$ or the {\em right cohomological dimension of $R$} is 
 \[ \cohdim(\rProj R) = \max \{ i \st H^i(\rProj R, \sh{M}) \neq 0 \mbox{ for some $\sh{M} \in \rQgr R$ } \}.\]
That is, $\cohdim (\rProj R)$ is the cohomological dimension of the functor $H^0(\rProj R, \blank)$.  
 We similarly define the {\em left cohomological dimension of $R$}, or $\cohdim (R \lProj)$.  
 
 If $R$ is  a finitely generated commutative graded $\kk$-algebra, then its cohomological dimension is finite and in fact bounded by the dimension of $\Proj R$.  The proofs of this are geometric, for example relying on \v{C}ech cohomology calculations, and do not generalize to the noncommutative situation.  Stafford and Van den Bergh have asked \cite[page~194]{SV}  if every connected graded noetherian ring has finite left and right cohomological dimension.  
 
In  this section, we answer  Stafford and Van den Bergh's question for geometric idealizers.   We prove:

\begin{theorem}\label{thm-cohdim}
Assume Assumption-Notation~\ref{assnotnot2}.  
If $R = R(X, \Lsh, \sigma, Z)$ is noetherian, then $R$ has finite left and right cohomological dimension.  
\end{theorem}

Thus, while we have seen that  the $\chi$ conditions  and the strong noetherian property are quite asymmetrical for geometric idealizers, cohomological dimension appears to behave better, at least in the (two-sided) noetherian case.  This, unsurprisingly, breaks down for non-noetherian rings, and we
  give an example of a right, but not left, noetherian ring with infinite right cohomological dimension.  Amusingly, this ring has finite left cohomological dimension.
 
To begin, we review Rogalski's results on the cohomological dimension of idealizers.  

\begin{proposition}\label{prop-R-triple}
{\em (\cite[Lemma~3.2]{R-idealizer})}
Let $B$ be a noetherian  connected graded finitely $\NN$-graded $\kk$-algebra, and let $I$ be  a graded right ideal of $B$ such that $R/I$ is infinite-dimensional.  
Assume that $B_R$ is finitely generated and $R/I$ is finite-dimensional.  Then there are isomorphisms of pairs
\beq \label{isom1}
R \lProj = (R \lQgr, \pi R) \cong (B \lQgr, \pi B) = B \lProj
\eeq
and
\beq \label{isom2}
\rProj R = (\rQgr R, \pi R) \cong (\rQgr B, \pi I).
\eeq
 \qed
\end{proposition}

Because of \eqref{isom1}, it is clear that $\cohdim(R \lProj) = \cohdim(B \lProj) = \dim X$, and this was observed by Rogalski.  We thus focus on calculating $\cohdim(\rProj R)$.

\begin{lemma}\label{lem-infinite-dim}
Assume Assumption-Notation~\ref{assnotnot2}.    
  Then $\cohdim(\rProj R)$ is infinite if and only if $\hd_X(\struct_Z)$ is infinite.
\end{lemma}
\begin{proof}
Let $I = \Gamma_*(\sh{I}) \subseteq B$.  
Since $(\rQgr B,  \pi I ) \cong (\struct_X \lMod, \sh{I})$, by  \eqref{isom2} $\cohdim (\rProj R)$ is infinite if and only if 
for any $k \geq 0$, there is some quasi-coherent $\sh{F}$ such that $\Ext^k_X(\mathcal{I}, \sh{F}) \neq 0$.

Suppose that $\hd_X(\struct_Z)$ and therefore $\hd_X(\mathcal{I})$ are infinite.  Thus for any $k > 0$, there is some $\mathcal{G}$ such that $\shExt^k_X (\mathcal{I}, \mathcal{G}) \neq 0$.  But let $\struct(1)$ be any very ample invertible sheaf on $X$; by \cite[III.6.9]{Ha} we may choose $n$ so that
$\Ext^k_X(\mathcal{I}, \mathcal{G}(n)) = H^0(X, \shExt^k_X (\mathcal{I}, \mathcal{G}) \otimes \struct(n)) \neq 0$.  Thus  
$\cohdim(\rProj {R}) \geq k$ and since $k$ was arbitrary, $\cohdim(\rProj {R})$ is infinite.

Now suppose that $\hd_X(\mathcal{I})$ is finite, say equal to $N$, and let $\mathcal{G}$ be an arbitrary coherent sheaf.   We apply  \eqref{ss-Gro} to obtain a spectral sequence 
\beq\label{LAURA}
 H^p (X, \shExt^q_X ( \mathcal{I}, \mathcal{G})) \Rightarrow \Ext^{p+q}_X (\mathcal{I}, \mathcal{G}).
\eeq
The left-hand side has nonzero terms only for $0 \leq p \leq \dim X$ and $0 \leq q \leq N$.  Thus if $p+q$ is large (in particular $p+q > N + \dim X$), then all the groups on the left-hand side are 0, and so \eqref{LAURA} collapses to 0.  Thus $\cohdim(\rProj R) \leq N  + \dim X$.
\end{proof}

  \begin{proof}[Proof of Theorem~\ref{thm-cohdim}]
If $R(X, \Lsh, \sigma, Z)$ is left noetherian, then by Proposition~\ref{prop-not-leftnoeth}, we have that $\{ \sigma^n Z\}_{n \geq 0}$ is homologically transverse to all $\sigma$-invariant subvarieties of $X$, and in particular, to the singular stratification of $X$.  Thus  by Lemma~\ref{lem-Mel}, $\hd_X (\struct_Z)$ is finite.  By Lemma~\ref{lem-infinite-dim}, $\cohdim (\rProj R)$ is finite.
\end{proof}

We now give the promised example of a right noetherian ring with infinite right cohomological dimension.

\begin{example}\label{eg-cohdim}
Assume that $\chrr \kk = 0$.  Let $Y$ be the cuspidal cubic and let $X = Y \times \PP^1$.  Let $\tau: \PP^1 \to \PP^1$ be the  automorphism $\tau([x:y]) = [x+y:y]$, and let $\sigma = 1 \times \tau \in \Aut_{\kk} X$.  Let $P$ be the singular point of $Y$ and let $Z = P \times [0:1] \in X$.  Let $\Lsh$ be any ample invertible sheaf on $X$, and let $R = R(X, \Lsh, \sigma, Z)$.   Since the numerical action of $\sigma$ is trivial, by \cite[Theorem~1.2]{Keeler2000} $\Lsh$ is $\sigma$-ample.

Now $Z$ is certainly of infinite order under $\sigma$, and   $R$ is right noetherian by Proposition~\ref{prop-SRN}.  On the other hand, $Z$ is contained in the singular locus of $X$, and so Proposition~\ref{prop-not-leftnoeth}(1) and Lemma~\ref{lem-contain-Tor} imply that  $R$ is not left noetherian.   Since $X$ is not regular at $Z$, we have that $\hd_X (\struct_Z)$ is infinite.   Lemma~\ref{lem-infinite-dim} implies that $\cohdim (\rProj R) = \infty$.  

We note that Proposition~\ref{prop-R-triple} implies that the left cohomological dimension of $R$ is 2.
\end{example}

{\bf Remark:}  Suppose that $R = R(X, \Lsh, \sigma,Z)$ is a left noetherian idealizer.  Together, Lemma~\ref{lem-infinite-dim} and Lemma~\ref{lem-Mel} imply that the right cohomological dimension of $R$ is bounded by $2 \dim X-1$.  We conjecture that in fact the left cohomological dimension of $R$ is precisely $\dim X$.  It is easy to see that $\cohdim(\rProj R) \geq \dim X$.

\section{Conclusion}\label{IDEALIZER-SUM}
Here we collect our results on geometric idealizers, and prove Theorem~\ref{thm-idealizermain} and its promised generalization.    Throughout, we make the following assumptions.
\begin{assumptions} \label{ass-final}
Let $X$ be a projective variety, let $\sigma$ be an automorphism of $X$, and  let $\Lsh$ be a $\sigma$-ample invertible sheaf on $X$.  Let $Z$ be a closed subscheme of $X$ such that for any irreducible component  $Y$ of $Z$, 
\[ \sigma^n(Y^{\red}) \not\subseteq Z\]
for $n \gg 0$.

Given this data, we let
\[ R = R(X, \Lsh, \sigma, Z).\]
Let $\sh{I} = \sh{I}_Z$ be the ideal sheaf of $Z$ on $X$.
\end{assumptions}

We note that since by Theorem~\ref{thm-rtnoeth} any noetherian right idealizer is up to a finite extension an idealizer at a scheme whose defining data satisfies Assumptions~\ref{ass-final}, these assumptions  are not unduly restrictive.

We now summarize our results.  

\begin{theorem}\label{thm-idealizersum}
Assume Assumptions~\ref{ass-final}.

{\rm(1)}  $R$ is right noetherian if and only if for any $x \in X$, the set $\{ n \geq 0 \st \sigma^n(x) \in Z \}$ is finite.

$(2)$ If $R$ is right noetherian, then $R$ is strongly right noetherian.

$(3)$ $R$ fails left $\chi_1$.  

$(4)$ If $\{ \sigma^n(Z)\}_{n \geq 0}$ is critically transverse, then $\{ (\sh{I} \Lsh_n)_{\sigma^n} \}_{n \geq 0}$ is a left and right ample sequence of bimodules, and $R$ is left noetherian.

$(5)$ $R$ is strongly left noetherian if and only if $ Z $ has pure codimension 1 and $\{ \sigma^n Z \}_{n \geq 0}$ is critically transverse.

$(6)$ Let $d = \codim Z$.  If $\{ \sigma^n Z\}_{n \leq 0}$ is critically transverse and either $d = \dim X$ or $X$ and $Z$ are both smooth, then $R$ satisfies right $\chi_{d-1}$.  If $R$ is noetherian, then $R$ fails right $\chi_d$.

$(7)$ If $R$ is noetherian, then $R$ has finite left and right cohomological dimension.

$(8)$  If $Z$ does not have pure codimension 1, then $R \otimes_{\kk} R$ is not left noetherian.
\end{theorem}

 We note that Theorem~\ref{thm-idealizermain} is a special case of Theorem~\ref{thm-idealizersum}.
\begin{proof}
(1) and (2)  are Proposition~\ref{prop-SRN}.  (3) is Proposition~\ref{prop-left-chi}.  (4) is Lemma~\ref{lem-rtample}, Proposition~\ref{prop-left-ample} and Proposition~\ref{prop-leftnoeth}.  (5) is Corollary~\ref{cor-SLN}.  
(6) is a special case of Proposition~\ref{prop-chi}, and (7) is Theorem~\ref{thm-cohdim}.  (8) is Proposition~\ref{prop-tensor}.  
\end{proof}

\bibliographystyle{amsalpha}       

\bibliography{biblio}  

\end{document}